\documentclass[11pt,a4paper]{article}

\usepackage[T1]{fontenc}
\usepackage{lmodern}
\usepackage[utf8]{inputenc}

\usepackage[left=2.0cm, top=2.45cm,bottom=2.45cm,right=2.0cm]{geometry}
\usepackage{mathtools,amssymb,amsthm,mathrsfs,calc,graphicx,xcolor,dsfont,tikz,bm}
\usepackage[british]{babel}
\usepackage{amsfonts}              

\usepackage[
	bookmarks=true,
	bookmarksnumbered=true,
	bookmarksopen=true,
	unicode=true,
	pdftoolbar=true,
	pdfmenubar=true,
	pdffitwindow=false,
	pdfstartview={FitH},
	pdftitle={},
	pdfauthor={},
	pdfsubject={},
	pdfcreator={},
	pdfproducer={},
	pdfkeywords={},
	pdfnewwindow=true,
	colorlinks=true,
	linkcolor=black,
	citecolor=black,
	filecolor=black,
	urlcolor=black]{hyperref}

\usepackage{cleveref}

\usepackage[T1]{fontenc}
\usepackage[deletedmarkup=xout]{changes}
\usepackage{selinput}
\numberwithin{equation}{section}
\numberwithin{figure}{section}

\makeatletter
	\def\@cite#1#2{[\textbf{#1}\if@tempswa , #2\fi]}	
	\def\@biblabel#1{[#1]}								
\makeatother

\newtheorem {theorem}{Theorem}[section]

\newtheorem {lemma}[theorem]{Lemma}

\newtheorem {remark}[theorem]{Remark}

\newcommand{\dint}{\textup{d}}
\newcommand{\pos}{\operatorname{pos}}

 \def\EE{\mathbb{E}}
 \def\PP{\mathbb{P}}
 \def\RR{\mathbb{R}}
 \def\SS{\mathbb{S}}
 \def\bZ{\mathbf{Z}}

 \def\bn{\mathbf{n}}
 \def\bu{\mathbf{u}}
 \def\by{\mathbf{y}}
 \def\bz{\mathbf{z}}
 \def\bx{\mathbf{x}}
 \def\bw{\mathbf{w}}
 \def\bv{\mathbf{v}}
 \def\be{\mathbf{e}}
 \def\bp{\mathbf{p}}

\def\I{\gamma}

\newcommand{\lspan}{\operatorname{span}}
\newcommand{\proj}{\operatorname{proj}}

\makeatletter
\let\@fnsymbol\@alph
\makeatother

\begin{document}

\title{\bfseries Spherical convex hull of random points on a wedge}

\author{Florian Besau\footnotemark[1],\; Anna Gusakova\footnotemark[2],\; Matthias Reitzner\footnotemark[3],\\
\;Carsten Sch\"utt\footnotemark[4],\; Christoph Th\"ale\footnotemark[5],\; Elisabeth M.\ Werner\footnotemark[6]}

\renewcommand{\thefootnote}{\fnsymbol{footnote}}

\footnotetext[1]{
    Technische Universität Wien, Austria. Email: florian.besau@tuwien.ac.at
}
\footnotetext[2]{
    M{\"u}nster University, Germany. Email: gusakova@uni-muenster.de
}
\footnotetext[3]{
    University of Osnabr\"uck, Germany. Email: matthias.reitzner@uni-osnabrueck.de
}
\footnotetext[4]{%
    University of Kiel, Germany. Email: schuett@math.uni-kiel.de
}
\footnotetext[5]{%
    Ruhr University Bochum, Germany. Email: christoph.thaele@rub.de
}
\footnotetext[6]{%
    Case Western Reserve University, USA. Email: elisabeth.werner@case.edu
}

\date{}

\maketitle

\begin{abstract}
Consider two half-spaces $H_1^+$ and $H_2^+$ in $\mathbb{R}^{d+1}$ whose bounding hyperplanes $H_1$ and $H_2$ are orthogonal and pass through the origin. The intersection $\mathbb{S}_{2,+}^d:=\mathbb{S}^d\cap H_1^+\cap H_2^+$ is a spherical convex subset of the $d$-dimensional unit sphere $\SS^d$, which contains a great subsphere of dimension $d-2$ and is called a spherical wedge. Choose $n$ independent random points uniformly at random on $\mathbb{S}_{2,+}^d$ and consider the expected facet number of the spherical convex hull of these points. It is shown that, up to terms of lower order, this expectation grows like a constant multiple of $\log n$. A similar behaviour is obtained for the expected facet number of a homogeneous Poisson point process on $\mathbb{S}_{2,+}^d$. The result is compared to the corresponding behaviour of classical Euclidean random polytopes and of spherical random polytopes on a half-sphere.


\end{abstract}

\section{Introduction}

One of the classical ways to construct a random polytope is based on taking a fixed convex body $K\subset\mathbb{R}^d$ and a sequence $(X_i)_{i\geq 1}$ of independent random points uniformly distributed in $K$. For $n\geq d+1$ we denote by
\begin{equation*}
    K_n:=[X_1,\ldots,X_n]
\end{equation*}
the convex hull of the random points $X_1,\ldots,X_n$. We may consider it as a random polytopal approximation of the set $K$, which approaches $K$, as the number $n$ of random points tends to infinity. There are several interesting random variables connected to the random polytopes $K_n$, some of which can be seen as a measure of the degree of approximation of $K$ achieved by $K_n$, while others give information on the combinatorial complexity of $K_n$. Usually one of the following notions is considered:
\begin{itemize}
    \item[(i)] the volume $V_d(K_n)$ of $K_n$, and, more generally, the $k$-th intrinsic volume $V_k(K_n)$ of $K_n$, $k=\{0,1,\ldots,d\}$,
    \item[(ii)] the number $f_{d-1}(K_n)$ of facets, that is, the $(d-1)$-dimensional faces, of $K_n$, and, more generally, the number $f_k(K_n)$ of $k$-dimensional faces of $K_n$, $k=\{0,1,\ldots,d-1\}$.
\end{itemize}
From now on, we will concentrate on first-order properties of the combinatorial structure of $K_n$, more precisely, on the expected number $\mathbb{E}f_k(K_n)$ of $k$-dimensional faces of $K_n$. For further background material on random polytopes we refer to the survey articles \cite{BaranySurvey,HugSurvey,ReitznerSurvey}.

\smallskip

It is well known that the asymptotic behavior of $\mathbb{E}f_k(K_n)$, as $n\to\infty$, depends on the geometry of the underlying convex body $K$. Indeed, if $K$ is of class $C_+^2$, that is, if $K$ has a boundary which is a $C^2$-submanifold of $\mathbb{R}^d$ with strictly positive Gaussian curvature $\kappa(x)$ at every point $x\in\partial K$, then
\begin{equation}\label{eqn:smooth}
    \mathbb{E}f_k(K_n) = c_{d,k}\Omega(K)n^\frac{d-1}{d+1}(1+o_d(1)),
\end{equation}
as $n\to\infty$, where $c_{d,k}$ is a constant only depending on $d$ and on $k$ and $\Omega(K)=\int_{\partial K}\kappa(x)^{1/(d+1)}\,{\rm d}x$ is known to be the affine surface area of $K$, see e.g.\ \cite[Theorem 4]{R05}. Here and in the following we use the \emph{little-o notation} for our error term, that is, by $f(n)=o(g(n))$ we mean that $\lim_{n\to\infty} f(n) / g(n) = 0$, where we assume that $g(n)>0$. An analogue of \eqref{eqn:smooth} in the spherical space was derived in \cite{BLW18} using tools from \cite{BFH:2010}.

\smallskip
On the other hand, if $K=P$ is a $d$-dimensional polytope, then
\begin{equation}\label{eq:Polytopes}
\mathbb{E}f_k(K_n) = \hat c_{d,k}{\rm flag}(P)(\log n)^{d-1}(1+o_d(1)),
\end{equation}
as $n\to\infty$, where $\hat c_{d,k}$ is another constant only depending on $d$ and on $k$, while ${\rm flag}(P)$ is the number of flags of $P$, that is, the number of chains $F_0\subset F_1\subset\ldots\subset F_{d-1}$, where for each $i\in\{0,1,\ldots,d-1\}$, $F_i$ is an $i$-dimensional face of $P$, see e.g.\ \cite[Theorem 8]{R05}. On different levels of generality these results can be found in \cite{RS63} for $d=2$, \cite{BB93} for $k\in\{0,d-1\}$ and \cite{R05} for general $d$ and $k$. See also \cite{BSW18,Schutt:1991} for a limit theorem similar to \eqref{eq:Polytopes} in the context of (weighted) floating bodies.

\smallskip
Let us also mention another model, which is often considered together with $K_n$. Let $\eta_\I$ be a Poisson point process in $\RR^d$ with intensity measure given by a constant multiple $\I>0$ of the Lebesgue measure restricted to some convex body $K\subset\RR^d$. We define the Poisson random polytope $K_{\eta_\I}$ as a convex hull of the Poisson point process $\eta_\I$. We remark in this context that the expected number of points of $\eta_\I$ equals $\I V_d(K)$, meaning that for a Poisson random polytope the quantity $\I V_d(K)$ plays the same role as the number $n$ for the classical random polytopes described at the beginning. 

\medskip
To motivate our results, let us recall the following setup, which has been introduced in \cite{BHRS17}. We assume that $(X_i)_{i\geq 1}$ is a sequence of independent random points uniformly distributed on the $d$-dimensional upper halfsphere $\mathbb{S}^d_+:=\mathbb{S}^d\cap\{x_{d+1}\geq 0\}\subset\mathbb{R}^{d+1}$. For $n\geq d+1$ we consider the spherical convex hull 
\begin{equation*}
    K_n^{(s)}:=[X_1,\ldots,X_n]_{\SS^d}
\end{equation*}
of the random points $X_1,\ldots,X_n$, which is defined as the intersection of their positive hull 
\begin{equation*}
    \pos(X_1,\ldots,X_n):=\{\lambda_1X_1+\ldots+\lambda_n X_n\colon \lambda_1,\ldots,\lambda_n\ge 0\}\subset \mathbb{R}^{d+1}
\end{equation*}
with the unit sphere $\mathbb{S}^d$. This convex hull $K_n^{(s)}$ is a spherical random polytope that approximates the half-sphere $\mathbb{S}^d_+$, as $n\to \infty$. Remarkably, the expected number of $k$-dimensional spherical faces of $K_n^{(s)}$ does not grow to infinity with $n\to\infty$ as it is the case for classical random polytopes in $\RR^d$. Instead we have that, without any renormalization,
\begin{equation}\label{eq:PolytopesSphere}
    \lim_{n\to\infty}\mathbb{E}f_{k}(K_n^{(s)}) = \tilde{c}_{d,k},
\end{equation}
where $\tilde{c}_{d,k}$ is a constant only depending on $d$ and $k$, see \cite{BHRS17} for $k\in\{0,d-1\}$ and \cite{KMTT} for general $k$. We remark that after gnomonic projection with respect to the north pole of the interior of the half-sphere $\mathbb{S}^d_+$, the spherical random polytope $K_n^{(s)}$ may be identified with the convex hull of random points in $\RR^d$, having a so-called beta-prime distribution in $\RR^d$ with parameter $\beta=\frac{d+1}{2}$, where the probability density of a beta-prime distribution on $\RR^d$ with parameter $\beta>d/2$ is given by
\begin{equation}\label{eq:betaprimedensity}
    \widetilde{f}_{d,\beta}(\bx) = \widetilde{c}_{d,\beta}(1+\|\bx\|^2)^{-\beta},\qquad x\in\RR^d,\quad 
    \widetilde{c}_{d,\beta}     := \frac{\Gamma(\beta)}{\pi^{\frac{d}{2}}\Gamma(\beta-\frac{d}{2})},
\end{equation}
see \cite{KMTT,KTZ}. 

The half-sphere $\mathbb{S}_+^d$ can be seen as a spherical convex polytope with a single facet and no other boundary structure, and similarly we want to think of $\RR^d$ as a $d$-dimensional convex `unbounded polytope' with a single facet at infinity.  Comparing the case of bounded polytopes in $\RR^d$ with the `unbounded polytope' we arrive at the following natural question:
\emph{Are there models for random polytopes that interpolate between the behaviour of \eqref{eq:Polytopes} and \eqref{eq:PolytopesSphere}?}

\smallskip
In order to answer this question we are focusing in this paper on the case $k=d-1$ and on the following construction, which generalizes the approach in \cite{BHRS17,KMTT}. Given $j\in\{1,\ldots,d\}$ hyperplanes $H_1,\ldots,H_j$ passing through the origin of $\mathbb{R}^{d+1}$ and are otherwise in general position we define the set
\begin{equation*}
    \mathbb{S}_{j,+}^d:=\mathbb{S}^d\cap H_1^+\cap\ldots\cap H_j^+,
\end{equation*}
where $H_i^+$ denotes the positive halfspace, bounded by the hyperplane $H_i$, $i\in\{1,\ldots,j\}$. Then $\mathbb{S}_{j,+}^d$ is a $d$-dimensional spherical convex subset of $\mathbb{S}^d$, which contains a great subsphere of dimension $d-j$, and its shape is determined by the angles between $H_1,\dotsc, H_j$. Let further $(X_i)_{i\geq 1}$ be independent random points uniformly distributed on $\mathbb{S}_{j,+}^d$ and for $n\geq d+1$ let $K_{n}^{(s,j)}$ be the spherical convex hull of $X_1,\ldots,X_n$. Note that for $j=1$, up to a rotation, $\mathbb{S}_{1,+}^{d}$ can be identified with $\mathbb{S}_+^d$ and the number of facets of the spherical random polytope $K_n^{(s,1)}$ has the same distribution as that of $K_n^{(s)}$ studied in \cite{BHRS17,KMTT}. On the other hand, intersecting the unit sphere with $d+1$ halfspaces one can think of $K_n^{(s,d)}$ after gnomonic projection as the convex hull of $n$ independent random points, having a beta-prime distribution with parameter $\beta=\frac{d+1}{2}$ restricted to the domain of a $d$-dimensional simplex in $\mathbb{R}^d$. In parallel to this model, to which we shall refer to as the binomial model, we also consider its analogue for Poisson point processes, the so-called Poisson model. Namely, let $\eta_{\I}$ be a Poisson point process in $\SS^d$ with intensity measure given by $x\mapsto \I {\bf 1}_{\mathbb{S}_{j,+}^d}(x)$. We define a spherical Poisson random polytope $K_{\eta_{\I}}^{(s,j)}$ as a spherical convex hull of $\eta_{\I}$.

We conjecture that the models just described provide a family of examples where the behavior of the expected facet number varies between \eqref{eq:PolytopesSphere} and \eqref{eq:Polytopes} as $j$ varies between $1$ and $d$. More precisely, we formulate the following conjecture, where in each case $o_d(1)$ stands for a dimension-dependent sequence which converges to zero, as $n\to\infty$.

\smallskip
\pagebreak[3]
\noindent \textbf{Conjecture:} \textit{For $j\in\{1,\ldots,d\}$ one has that}
\begin{align*}
    \EE f_{d-1}(K_{n}^{(s,j)}) &= c_{d,j}\, (\log n)^{j-1}(1+o_{d}(1)), \qquad \text{as $n\to \infty$;}\\
    \EE f_{d-1}(K_{\eta_{\I}}^{(s,j)}) &= \bar{c}_{d,j}\, (\log \I)^{j-1}(1+o_{d}(1)),\qquad \text{as $\I\to \infty$,}
\end{align*}
\textit{where $c_{d,j}, \bar{c}_{d,j}$ are constants that depend only on $d$ and $j$. In particular, the first-order asymptotic expansion does not depend on the actual angles between $H_1,\dotsc,H_j$, assuming they are in general position. We also conjecture that actually $c_{d,j}=\bar{c}_{d,j}$ and that the error terms $o_d(1)$ depend on the angles between the hyperplanes $H_1,\ldots,H_j$.}
\goodbreak

\medskip
In this article we prove a special case of this conjecture. Namely, we consider the case $j=2$ and assume that the angle $\alpha(H_1, H_2)$ between the hyperplanes $H_1$ and $H_2$ is a right angle. The resulting set $\mathbb{S}_{2,+}^d$ is called a spherical wedge in this paper. Our main result is the following theorem.

\begin{theorem}\label{thm:main}
Let $K=\mathbb{S}_{2,+}^d$ and suppose that $\alpha(H_1, H_2)=\frac{\pi}{2}$. Then there exists a constant $c_{d,2}>0$ only depending on the space dimension $d$ such that 
\begin{align*}
    \EE f_{d-1}(K_{n}^{(s,2)}) &= c_{d,2}\,(\log n)\, (1+o_{d}(1)),\qquad \text{as $n\to \infty$},
\intertext{and}
    \EE f_{d-1}(K_{\eta_{\I}}^{(s,2)}) &= c_{d,2}\,(\log \I)\, (1+o_{d}(1)),\qquad \text{as $\I\to \infty$}.
\end{align*}
\end{theorem}

\begin{figure}[t]
    \centering
    \begin{tikzpicture}[scale=0.80,transform shape]
        \begin{scope}
            \clip (-5.7,-2.2) rectangle (5.7,5.7);;
            \node at (0,0) {\includegraphics[width=18cm]{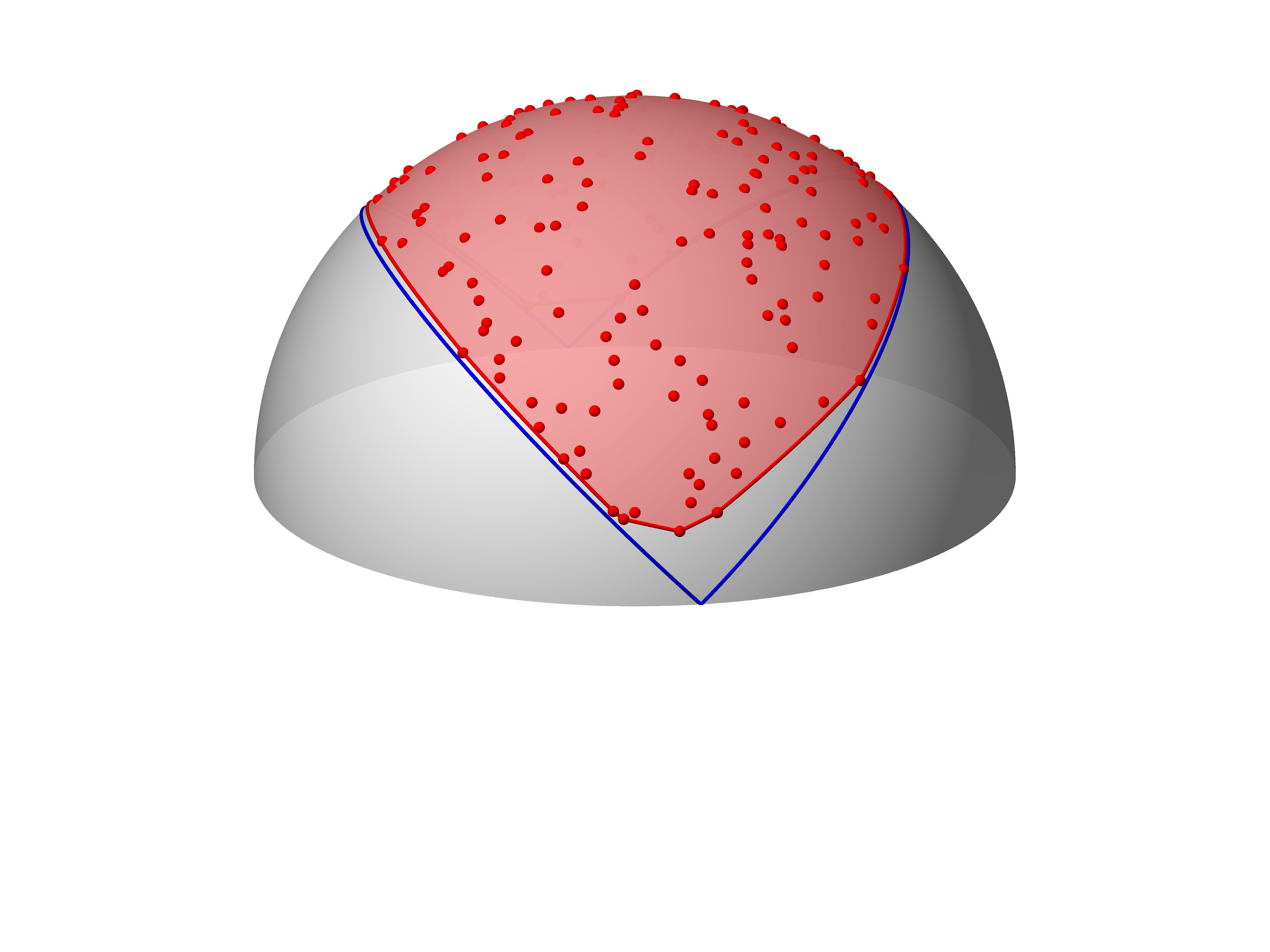}};
        \end{scope}

        \begin{scope}[yshift=-4cm]
            \clip (-8.2,-1.2) rectangle (8.2,1.2);
            \node at (0,0) {\includegraphics[width=15cm]{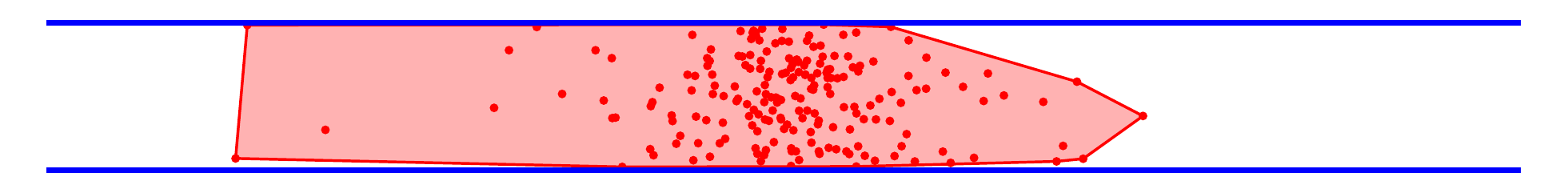}};
        \end{scope}
    \end{tikzpicture}
    
    \caption{The upper panel shows a random spherical polygon in the spherical wedge of dimension two. The same random spherical polygon is shown in the lower panel after gnomonic projection in the center of the spherical wedge.}
    \label{fig:random_wedge}
\end{figure}

We were not able to determine the constant $c_{d,2}$ explicitly. However, we have that
\begin{equation*}
        c_{d,2} = \frac{1}{d} 2^{d-1} \omega_{d-1} A_d,
\end{equation*}
where $A_d$ is the constant defined by \eqref{eqn:A_d} below, whose value is not known. Moreover, for $d=2$ we can determine $A_d$ explicitly, which leads to $c_{2,2} = \frac{4}{3}$, see Remark \ref{rem:A2}. This in turn yields that
\begin{align*}
    \mathbb{E}f_{1}(K_{n}^{(s,2)}) &= \frac{4}{3} (\log n) (1+ o_d(1)), \qquad \text{as $n\to\infty$},
\end{align*}
for $d=2$. This should be compared to the expected number of edges (or vertices) of a random polygon $K_n^{(\ell)}$ defined as the convex hull of $n$ independent and uniform random points in a planar polygon with $\ell\geq 3$ edges. For this model it is known from \cite[Satz 1]{RS63} that
\begin{align*}
    \mathbb{E}f_{1}(K_n^{(\ell)}) &= \frac{2\ell }{3} (\log n) (1+ o_d(1)),\qquad \text{as $n\to\infty$},
\end{align*}
see also \cite[Equation (1.6)]{BB93} for an extension to higher dimensions.
Thus, the leading-order asymptotic behavior of the expected edge number of $\mathbb{E}f_{1}(K_{n}^{(s,2)})$ is smaller than $\mathbb{E}f_{1}(K_n^{(\ell)})$ for any $\ell\geq 3$ and would be similar to a hypothetical $2$-dimensional convex polygon with only $\ell=2$ edges. A similar result holds for the Poisson model as well.

While random spherical polytopes in the upper half-sphere, as discussed above, have been studied in \cite{BHRS17} using tools from spherical integral geometry, it has turned out to be a fruitful idea to first apply a gnomonic projection with respect to the center of the half-sphere. The investigation of random spherical polytopes on half-spheres then turns into a study of a random polytope model in Euclidean space in which the points are distributed according to a so-called beta-prime distribution (see \cite{BLW18, KMTT}). In this paper we shall also use the gnomonic projection, this time with respect to the centre of the spherical wedge, which maps the random spherical polytope we consider to a random polytope defined in an bi-infinite strip in the Euclidean space, where the width of the strip is determined by the angle of the spherical wedge. The generating points are then distributed according to a restricted beta-prime distribution in Euclidean space, see Figure \ref{fig:random_wedge}. 
However, and in contrast to the half-sphere model, we were neither able to carry out the asymptotic analysis of the expected facet number entirely in the spherical model, nor in the corresponding Euclidean model in the strip. Instead we shall switch between both models and employ the tools developed in this paper and methods from both spherical \textit{and} Euclidean integral geometry.

\bigskip

The remaining parts of this paper are structured as follows. In Section \ref{sec:Prelim} we introduce some notation. The proof of Theorem \ref{thm:main} is divided into four steps, which are the content of Sections \ref{sec:ProofStep1}--\ref{sec:ProofStep4}. In the appendix we derive some analytic estimates used in our main argument.

\section{Preliminaries}\label{sec:Prelim}


By $\sigma_{d}$ we denote the spherical Lebesgue measure on the $d$-dimensional unit sphere $\SS^{d}$, normalized in such a way that 
\begin{equation*}
    \sigma_d(\SS^d)=\omega_{d+1}:=\frac{2\pi^\frac{d+1}{2}}{\Gamma\big(\frac{d+1}{2}\big)}.
\end{equation*}
Given vectors $\bx,\by\in\RR^{d+1}$ we denote by $\bx\cdot\by$ their scalar product and by $\|\bx\|$ the Euclidean norm of vector $\bx$. Given a vector $\bu\in\RR^{d+1}\setminus\{0\}$ and a number $t\ge 0$ we define a hyperplane
\begin{equation*}
    H_t(\bu):=\{\bx\in\RR^{d+1}\colon \bx\cdot\bu=t\}
\end{equation*}
in $\RR^{d+1}$. We define two closed halfspaces
\begin{equation*}
    H_t(\bu)^+:=\{\bx\in\RR^{d+1}\colon \bx\cdot\bu\ge t\}\qquad\text{and}\qquad H_t(\bu)^-:=\{\bx\in\RR^{d+1}\colon \bx\cdot\bu\leq t\},
\end{equation*}
bounded by the hyperplane $H_t(\bu)$. In case $t=0$ we simplify our notation and just write $H(\bu):=H_0(\bu)$. Moreover, if the dependence of the normal vector is inessential, we will omit it and simply write $H$, $H^+$ and $H^-$.

\smallskip
Given vectors $\bv_1,\ldots,\bv_m\in\RR^{d+1}$ we denote by $\lspan\{\bv_1,\ldots,\bv_m\}$ their linear span, which is the smallest linear subspace, containing all of them. Denote by $\be_1,\ldots,\be_{d+1}$ the standard orthonormal basis of $\RR^{d+1}$. In what follows we will often identify $\lspan\{\be_1,\ldots,\be_k\}$ with $\RR^k$. Given a linear $m$-dimensional subspace $M\subset\RR^k$, $m\leq k$ we denote by $\proj_M:\RR^k\mapsto M$ the orthogonal projection operator onto $M$ and by $M^{\perp}$ the orthogonal complement of $M$.

Given a vector $\bu\in\SS^d$ we denote by $\SS^d_{\bu}:=\{\bx\in\SS^d\colon \bx\cdot\bu >0\}$. The gnomonic projection operator $g^d_{\bu}: \SS^d_{\bu}\mapsto H(\bu)$ with respect to point $\bu$ is defined by
\begin{equation}\label{eq:gnomonicproj}
g^d_{\bu}(\bv)=\frac{\bv}{\bu\cdot\bv}-\bu.
\end{equation}
The inverse of the gnomonic projection $(g^d_{\bu})^{-1}:H(\bu)\mapsto \SS^d_{\bu}$ is given by
\begin{equation}\label{eq:gnomonicprojinv}
(g^d_{\bu})^{-1}(\bx)=\frac{\bx+\bu}{\|\bx+\bu\|},
\end{equation}
see \cite[Section 4]{BesauSchuster} for further background material about the gnomonic projection.

\section{Proof of Theorem \ref{thm:main}, Step 1: Reduction}\label{sec:ProofStep1}

We consider first the random polytope model $K_{\eta_{\I}}^{(s,2)}$ based on the Poisson point process $\eta_\I$. Applying the multivariate Mecke formula for Poisson point processes \cite[Corollary 3.2.3]{SW} we find that
\begin{align*}
\mathbb{E}f_{d-1}(K_{\eta_{\I}}^{(s,2)})
    &= \frac{1}{d!}\mathbb{E}\sum_{(\bx_1,\ldots,\bx_d)\in(\eta_{\I})^d_{\neq}}{\bf 1}\{\bx_{1},\ldots,\bx_{d}\text{ generate a facet of }K_{\eta_{\I}}^{(s,2)}\}\\
    &=\frac{\I^d}{d!}\int_{\mathbb{S}_{2,+}^d}\cdots\int_{\mathbb{S}_{2,+}^d}\mathbb{P}(\bx_1,\ldots,\bx_d\text{ generate a facet of }K_{\eta_{\I}\cup\{x_1,\ldots,x_d\}}^{(s,2)})\,\sigma_d(\dint \bx_1)\ldots \sigma_d(\dint \bx_d),
\end{align*}
where the sum runs over all $d$-tuples of distinct points of $\eta_\I$. Next we apply the Blaschke--Petkantschin formula from spherical integral geometry \cite[Lemma 3.2]{BHRS17}, which states that
\begin{equation}\label{eqn:BP-formula}
\begin{aligned}
    &\int_{\mathbb{S}^d}\cdots \int_{\mathbb{S}^d} f(\bx_1,\dotsc,\bx_d)\, \sigma_d(\dint \bx_1)\ldots \sigma_d(\dint \bx_d)= \frac{\omega_{d+1}}{2} \int_{G(d+1,d)}\left[\int_{\mathbb{S}^d\cap H}\cdots \int_{\mathbb{S}^d\cap H} f(\bx_1,\dotsc,\bx_d)\right.\\
    & \hspace{6cm}\times \left.\nabla_d(\bx_1,\dotsc,\bx_d)\, \sigma_{d-1}(\dint \bx_1)\ldots \sigma_{d-1}(\dint \bx_d)\right] \nu_d(\dint H),
\end{aligned}
\end{equation}
where $f:\mathbb{S}^d\to \RR$ is a Borel measurable function, $G(d+1,d)$ is the Grassmannian of $d$-dimensional linear subspaces of $\mathbb{R}^{d+1}$ endowed with rotation invariant Haar probability measure $\nu_d$ and $\nabla_{d}(\bx_1,\ldots,\bx_d)$ is the Euclidean volume of the $d$-dimensional parallelotope spanned by $\bx_1,\ldots,\bx_d$. This formula is a special case of more general kinematic formula obtained in \cite{AZ91} and can be easily derived from the classical linear Blaschke--Petkantschin formula \cite[Theorem 7.2.1]{SW}.

By applying \eqref{eqn:BP-formula} and using the counting property of a Poisson point process we derive that
\begin{align*}
   \mathbb{E}f_{d-1}(K_{\eta_{\I}}^{(s,2)})
   &= \frac{\I^d\omega_{d+1}}{2d!}\int_{G(d+1,d)}
    \bigg(\int_{\mathbb{S}_{2,+}^d\cap H}\cdots\int_{\mathbb{S}_{2,+}^d\cap H}\nabla_{d}(\bx_1,\ldots,\bx_d)\,\sigma_{d-1}(\dint \bx_1)\ldots\sigma_{d-1}(\dint \bx_d)\bigg)\\
    &\qquad \times \big(\exp(-\I\sigma_d(\mathbb{S}_{2,+}^d\cap H^+))+\exp(-\I\sigma_d(\mathbb{S}_{2,+}^d\cap H^-))\big)\,\nu_d(\dint  H)\\
    &= \frac{\I^d}{2d!}\int_{\SS^d}
    \bigg(\int_{\mathbb{S}_{2,+}^d\cap H(\bz)}\cdots\int_{\mathbb{S}_{2,+}^d\cap H(\bz)}\nabla_{d}(\bx_1,\ldots,\bx_d)\,\sigma_{d-1}(\dint \bx_1)\ldots\sigma_{d-1}(\dint \bx_d)\bigg)\\
    &\qquad\times \big(\exp(-\I\sigma_d(\mathbb{S}_{2,+}^d\cap H^+(\bz)))+\exp(-\I\sigma_d(\mathbb{S}_{2,+}^d\cap H^+(-\bz)))\big)\,\sigma_d(\dint \bz),
\end{align*}
where in the last step we used the fact that each unit normal vector $\bz\in\SS^d$ determines a hyperplane $H\in G(d+1,d)$.
We set
\begin{align*}
    I_1(\bz)
    &:=\int_{\mathbb{S}_{2,+}^d\cap H(\bz)}\cdots\int_{\mathbb{S}_{2,+}^d\cap H(\bz)}
        \nabla_{d}(\bx_1,\ldots,\bx_d)\,\sigma_{d-1}(\dint \bx_1)\ldots\sigma_{d-1}(\dint \bx_d),\\
    I_2(\bz)
    &:=\sigma_d(\mathbb{S}_{2,+}^d\cap H^+(\bz)),
\end{align*}
and write
\begin{align*}
   \mathbb{E}f_{d-1}(K_{\eta_{\I}}^{(s,2)})
   &= \frac{\I^d}{2d!}\int_{\SS^d} I_1(\bz)\Big(\exp(-\I I_2(\bz))+\exp(-\I I_2(-\bz))\Big)\,\sigma_d(\dint \bz)\\
    &=\frac{\I^d}{d!}\int_{\SS^d} I_1(\bz)\exp(-\I I_2(\bz))\,\sigma_d(\dint \bz),
\end{align*}
where we used the fact that antipodal mapping $\bz\mapsto -\bz$ is isometric on $\mathbb{S}^d$ and we have $I_1(-\bz) = I_1(\bz)$.

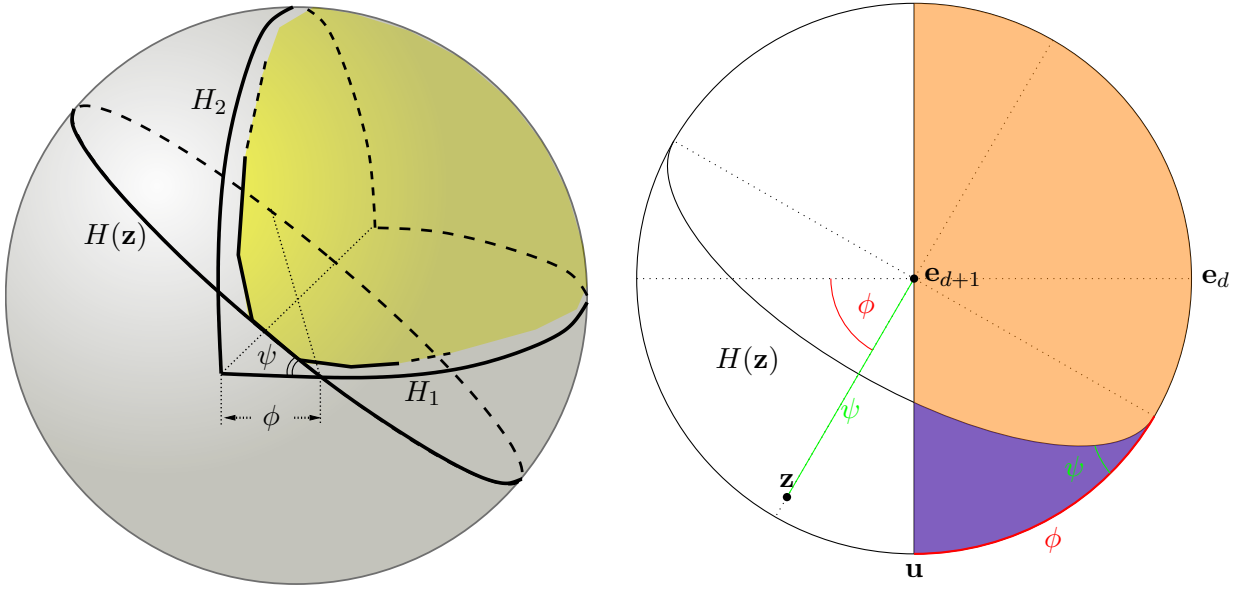
\begin{figure}[t]
\begin{picture}(150,200)
\centering
\put(20,0){\includegraphics[scale=0.45]{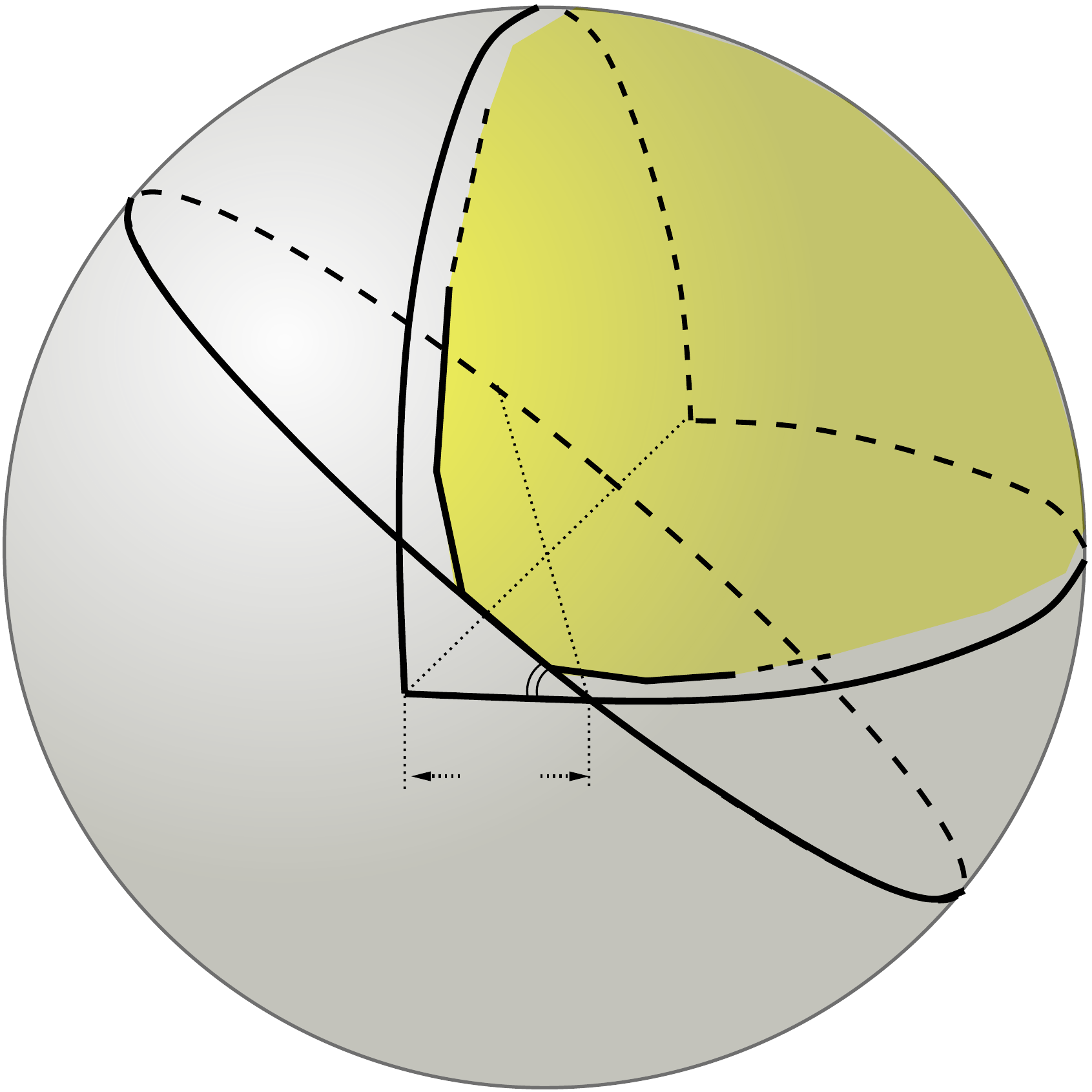}}
\put(115,84){$\psi$}
\put(117,61){$\phi$}
\put(50,129){$H(\bz)$}
\put(170,70){$H_1$}
\put(90,180){$H_2$}
\end{picture} 
    \hspace{3.5cm}
    \begin{tikzpicture}[scale=3.65]
        \draw (0,0) circle(1);
        \draw[dotted] (-1,0) -- (1,0) node[right] {$\be_d$};
        \draw[dotted] (60:1) -- (240:1);
        \draw (0,1) -- (0,-1) node[below] {$\bu$};
        \draw[rotate=-30] (1,0) arc(0:180:1 and -0.4);
        \draw[dotted] (-30:1) -- (150:1);
        \draw[red] (180:0.3) arc (180:240:0.3);
        \node[red] at (210:0.2) {$\phi$};
        
        \fill[orange,opacity=0.5] (0,-1) arc(-90:90:1) -- cycle;
        \fill[blue, opacity=0.5] (0,-1) arc(-90:-30:1) [rotate=-30]arc(0:77:1 and -0.4)--cycle;
        
        \draw[red, thick] (0,-1) arc(-90:-30:1);
        \node[red, below] at (-60:1) {$\phi$};
        
        \draw[green] (-45:1) arc(-140:-162:0.3);
        \node[green] at (-50:0.9) {$\psi$};

        \draw[green] (0,0) -- (-120:{sqrt(1-0.4^2)}) node[midway, below]{$\psi$}; 
        
        \node at (-0.6,-0.3) {$H(\bz)$};
        \fill (0,0) circle(0.015) node[right]{$\be_{d+1}$};
        \fill (-120:{sqrt(1-0.4^2)}) circle(0.015) node[above] {$\bz$};
        
    \end{tikzpicture}
    \caption{A spherical wedge with opening angle $\pi/2$ together with the great hyperspheres induced by $H_1,H_2$ and $H(\bz)$ (left) and its orthogonal projection to the hyperplane $\be_{d+1}^{\perp}$ (right).}
    \label{fig:my_label}
\end{figure}

As a next step we introduce spherical coordinates. For this we choose an orthonormal basis $\be_1,\ldots,\be_{d+1}$ of $\RR^{d+1}$ in such a way, that $H_1=H(\be_{d+1})$ and $H_2=H(\be_d)$. Then
\begin{equation*}
    \mathbb{S}_{2,+}^d 
    = \mathbb{S}^d \cap H^+(\be_{d}) \cap H^+(\be_{d+1}) 
    = \{ \bw\in\mathbb{S}^d : w_d\geq 0, w_{d+1}\geq 0\},
\end{equation*}
and we consider spherical coordinates, which are described by the map
\begin{align*}
    Z: [0,\pi)\times [0,\pi)\times \SS^{d-2}&\to \SS^d\\
    (\phi,\psi,\bu)&\mapsto (\sin\phi)(\sin\psi) \bu - (\cos\phi)(\sin\psi) \be_d - (\cos \psi) \be_{d+1},
\end{align*}
where we identify $\mathbb{S}^d\cap H(\be_{d})\cap H(\be_{d+1})$ with $\SS^{d-2}$.
Hence, for a Borel measureable function $f:\SS^d\to \RR$, we have that
\begin{equation*}
    \int_{\SS^d} f(\bz)\, \sigma_d(\dint \bz) 
    = \int_{\SS^{d-2}} \int_{0}^\pi\int_{0}^\pi f(Z(\phi,\psi,\bu)) \, (\sin\psi)^{d-1} \dint\psi\, (\sin\phi)^{d-2} \dint\phi \, \sigma_{d-2}(\dint \bu).
\end{equation*}

Now, due to the symmetry of the wedge $\SS^d_{2,+}$ with respect to $\lspan\{\be_{d},\be_{d+1}\}$, we have that 
\begin{align*}
    I_{1}(Z(\phi,\psi,\bu))
    &=I_{1}(Z(\pi-\phi,\pi-\psi,\bu)),
\intertext{and}
    I_2(Z(\phi,\psi,\bu))
    &= \sigma_d(\mathbb{S}_{2,+}^d)-I_2(Z(\pi-\phi,\pi-\psi,\bu)),
\end{align*}
for $\phi\in[0,\pi)$, $\psi\in\big[0,\frac{\pi}{2})$.
Thus, we may restrict ourselves to $\phi\in[0,\pi)$ and $\psi\in[0,\frac{\pi}{2})$, see Figure \ref{fig:my_label}.

Applying this transformation to the last representation for $\mathbb{E}f_{d-1}(K_{\eta_\I}^{(s,2)})$ yields
\begin{equation}\label{eq:Eq2}
\begin{aligned}
    \mathbb{E}f_{d-1}(K_{\eta_\I}^{(s,2)}) 
    &= \frac{\I^d}{d!}\int_{\SS^{d-2}} \int_{0}^{\pi}\int_{0}^{\pi/2} 
         I_1(Z(\phi,\psi,\bu)) \\
    &\qquad\qquad\times \left(\exp\!\big(-\I I_2(Z(\phi,\psi,\bu))\big) 
        + \exp\!\big(-\I\big[\sigma_d(\mathbb{S}_{2,+}^d) - I_2(Z(\phi,\psi,\bu))\big]\big)\right)\\ 
    &\qquad\qquad\times (\sin \phi)^{d-2} (\sin \psi)^{d-1}\, \dint \psi\, \dint \phi \, \sigma_{d-2}(\dint \bu).
\end{aligned}
\end{equation}

\begin{figure}[t]
\centering
\begin{tikzpicture}[scale=5]
    
    \fill[opacity=0.3] (0,0) arc(-90:-60:2 and 1) arc(65:79.9:4 and 2) arc(170:180:1 and 2) -- cycle;
    \fill[opacity=0.3] (0,0) arc(-90:-81.6:2 and 1) arc(10.1:25: 2 and 4) arc(150:180:1 and 2) -- cycle;

    \draw (0,0) arc (-90:-45:2 and 1) node[right] {$H_1=H(\be_{d+1})$};
    \draw (0,0) arc (-90:-95:2 and 1);
    
    \draw (0,0) arc (180:135: 1 and 2) node[above] {$H_2=H(\be_{d})$};
    \draw (0,0) arc (180:185: 1 and 2);
    
    \draw (0,1)++ (-60:2 and 1) arc (65:79.9: 4 and 2);
    \draw[thick] (1,0)++ (150:1 and 2) arc (25:10.1: 2 and 4) node[midway,right] {$\beta$};
    
    \draw (0,0) -- (1,1) node[right] {$H_0$};
    \draw (5:1.3) arc (5:85:1.3);
    
    \draw[thick] (0,0) arc (-90:-81.6:2 and 1) node[below,midway, yshift=0.1cm] {$\phi$};
    \draw (0,1)++(-83:2 and 1) arc (180:97:0.05) node[left, midway, yshift = 0.1cm, xshift=0.1cm] {$\psi$};
    
    \draw[thick] (0,0) arc (180:150: 1 and 2) node[left,midway] {$\widetilde{\phi}$};
    \draw (1,0)++ (153:1 and 2) arc (-110:-55:0.05) node[midway,below,xshift=-0.03cm] {$\widetilde{\psi}$};
    
    \draw (0,1)++ (-63:2 and 1) arc (200:145:0.05) node[above right, yshift=-0.1cm] {$\widetilde{\psi}$};
    
    \fill[black] (0,0) circle (0.01) node[below left] {$\bu$};
    \fill[black] (0,1)++(-50.3:2 and 1) circle(0.01) node[below right] {$\be_{d}$};
    \fill[black] (1,0)++(140.3:1 and 2) circle(0.01) node[above left] {$\be_{d+1}$};
    \fill[black] (0.92,0.92) circle (0.01) node[left, xshift=-0.1cm] 
        {$\frac{\be_{d}+\be_{d+1}}{\|\be_{d}+\be_{d+1}\|}$};
    
    \draw (0.22,0.22) arc(-140:-75:0.05) node[midway,below, xshift=-0.05cm] {$\alpha$};
    \draw (0,1)++(-88:2 and 1) arc (-10:54:0.05) node[midway, right, yshift=0.15cm] {$\frac{\pi}{4}$};
\end{tikzpicture}
\caption{The relation between $(\phi,\psi)$ and $(\widetilde{\phi},\widetilde{\psi})$ is derived by reflection about $H_0$.}\label{fig:symmetry}
\end{figure}
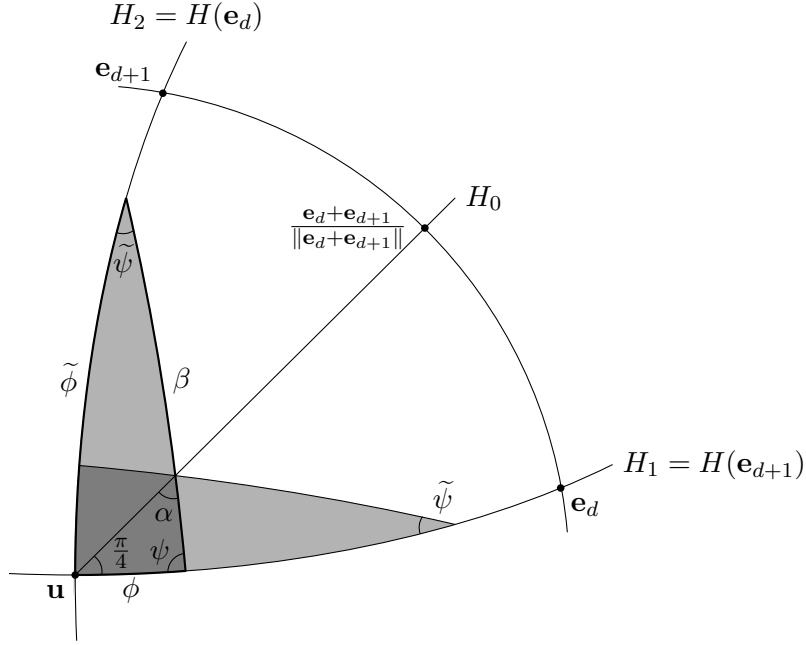

Finally, for $\phi,\psi\in[0,\pi/2)$ we may reflect with respect to the hyperplane $H_0 := \lspan\{\be_d+\be_{d+1}\} + (H(\be_{d})\cap H(\be_{d+1}))$ and switch the role of the hyperplanes $H_1=H(\be_{d+1})$ and $H_2=H(\be_{d})$. This operation is an isometry on $\SS^d$ and maps $(\phi,\psi)\in[0,\pi/2)\times [0,\pi/2)$ to $(\widetilde{\phi},\widetilde{\psi})\in[0,\pi/2)\times [0,\pi/2)$,
where $\widetilde{\phi},\widetilde{\psi}$ are determined due to Napier's rules for right spherical triangles by
\begin{equation*}
    \tan\widetilde{\phi} = (\tan \psi)(\sin \phi),\qquad
    \tan \phi = (\tan\widetilde{\psi})(\sin\widetilde{\phi}),
\end{equation*}
see Figure \ref{fig:symmetry}.
The mapping $G(\widetilde{\phi},\widetilde{\psi})=(\phi,\psi) $ does not change the value of $I_1$ and $I_2$, i.e., $I_1\circ G = I_1$ and $I_2\circ G = I_2$, for $\widetilde{\phi},\widetilde{\psi}\in[0,\pi/2)$.
Furthermore, if $\tan \psi > 1/\cos \phi$, then
\begin{equation*}
    \tan \widetilde{\psi} 
    = \frac{\tan\phi}{\tan\widetilde{\phi}} \frac{1}{\cos\widetilde{\phi}} 
    = \frac{1}{(\tan\psi)(\cos\phi)} \frac{1}{\cos\widetilde{\phi}} < \frac{1}{\cos\widetilde{\phi}}.
\end{equation*}
This condition arises from the observation that for the normal $\bn_0=\frac{1}{\sqrt{2}}(\be_{d+1}-\be_{d})$ of $H_0$ the angle $\alpha:=\alpha(\phi,\psi)$ between $Z(\phi,\psi,\bu)$ and $\bn_0$, i.e.,
\begin{equation*}
    \cos \alpha = z\cdot\bn_0 = \frac{1}{\sqrt{2}} \Big((\sin \psi)(\cos \phi)-\cos\psi\Big),
\end{equation*}
satisfies $\alpha+(\alpha\circ G)=\pi$ and
\begin{equation*}
    \alpha(\phi,\psi) <\frac{\pi}{2} \quad \Longleftrightarrow \quad \tan \psi > \frac{1}{\cos \phi}.
\end{equation*}

We calculate
\begin{equation*}
    |\det \nabla G| = \frac{\tan \widetilde{\psi}}{\sqrt{1+(\tan\widetilde{\psi})^2(\sin\widetilde{\phi})^2}}
\end{equation*}
and
\begin{equation*}
    \sin \phi = \frac{(\tan\widetilde{\psi})(\sin\widetilde{\phi})}{\sqrt{1+ (\tan\widetilde{\psi})^2 (\sin \widetilde{\phi})^2}},\qquad
    \sin \psi = (\cos \widetilde\psi) \sqrt{1+ (\tan\widetilde{\psi})^2 (\sin \widetilde{\phi})^2}.
\end{equation*}
Hence,
\begin{align*}
    (\sin\phi)^{d-2} (\sin\psi)^{d-1} |\det\nabla G| 
    = (\sin \widetilde{\phi})^{d-2} (\sin\widetilde{\psi})^{d-1}.
\end{align*}
This yields
\begin{align*}
    &\int_{0}^{\pi/2} \int_{0}^{\pi/2} 
        I_1 \exp(-\I I_2) (\sin\phi)^{d-2} (\sin\psi)^{d-1}
            \, \dint \psi\, \dint\phi\\
    &\qquad = \int_{0}^{\pi/2} \int_{0}^{\arctan(1/\cos \phi)} 
        I_1 \exp(-\I I_2) (\sin \phi)^{d-2} (\sin \psi)^{d-1} \, \dint \psi \,\dint \phi\\
    &\qquad\qquad + \int_{0}^{\pi/2} \int_{0}^{\arctan(1/\cos \widetilde{\phi})} 
        (I_1\circ G) \exp(-\I (I_2\circ G)) (\sin \widetilde{\phi})^{d-2} (\sin \widetilde{\psi})^{d-1} 
            \, \dint \widetilde{\psi}\, \dint \widetilde{\phi}\\
    &\qquad = 2\int_{0}^{\pi/2} \int_{0}^{\arctan(1/\cos\phi)} 
        I_1 \exp(-\I I_2) (\sin\phi)^{d-2} (\sin\psi)^{d-1}\, \dint \psi\, \dint\phi.
\end{align*}

Since for all $\varepsilon<1/2$, we have that
\begin{equation*}
    \{(\phi,\psi)\in[0,\pi/2)^2 : \tan\psi \leq 1/\cos\phi\} 
    \subset [0,\pi/2)^2 \setminus \Big([0,\varepsilon)\times [3\pi/8,\pi/2)\Big),
\end{equation*}
see Figure \ref{fig:regions}, together with \eqref{eq:Eq2} we conclude that for any such $\varepsilon$,
\begin{align}\label{eqn:lower}
    \mathbb{E}f_{d-1}(K_{\eta_\I}^{(s,2)}) 
    &\geq \frac{2}{d!}\I^d\int_{\SS^{d-2}} \int_{0}^\varepsilon \int_{0}^\varepsilon
            I_1 \exp(-\I I_2)\, (\sin\phi)^{d-2}(\sin\psi)^{d-1} \, \dint\psi\,\dint\phi\,\sigma_{d-2}(\dint\bu)
\end{align}
and
\begin{align}
 &\mathbb{E}f_{d-1}(K_{\eta_\I}^{(s,2)}) \notag\\
    &\leq \frac{2}{d!} \I^d\int_{\SS^{d-2}} \int_{0}^\varepsilon \int_{0}^\varepsilon
            I_1 \exp(-\I I_2)\, (\sin\phi)^{d-2}(\sin\psi)^{d-1} 
                \, \dint\psi\,\dint\phi\,\sigma_{d-2}(\dint\bu)\notag \\ \notag
    &\qquad +\frac{2}{d!} \I^d\!\int_{\SS^{d-2}}\int_0^{\varepsilon} \int_{\varepsilon}^{3\pi/8}
            I_1 \exp(-\I I_2)\, (\sin\phi)^{d-2}(\sin\psi)^{d-1} 
                \, \dint\psi\,\dint\phi\,\sigma_{d-2}(\dint\bu)\\ \notag
    &\qquad +\frac{2}{d!} \I^d \!\int_{\SS^{d-2}}\int_{\varepsilon}^{\pi} \int_{0}^{\pi/2}
            I_1 \exp(-\I I_2)\, (\sin\phi)^{d-2}(\sin\psi)^{d-1} 
                \, \dint\psi\,\dint\phi\,\sigma_{d-2}(\dint\bu)\\ \label{eqn:upper}
    &\qquad + \frac{1}{d!} \I^d\!\int_{\SS^{d-2}} \int_0^{\pi} \int_{0}^{\pi/2}\!
            I_1 \exp(-\I [\sigma_d(\SS_{2,+}^d)-I_2])\, (\sin\phi)^{d-2}(\sin\psi)^{d-1} 
                \, \dint\psi\,\dint\phi\,\sigma_{d-2}(\dint\bu).
\end{align}

\begin{figure}[t]
\centering
\begin{tikzpicture}[scale=2]
    \draw[->] (-0.25,0) -- (2.25,0) node[below] {$\phi$};
    \draw[->] (0,-0.25) -- (0,2.25) node[left] {$\psi$};
    
    \fill[domain=0:2, smooth, variable=\x,samples=500,opacity=0.3] plot ({\x}, {atan(1/cos(\x*45))/90*2}) -- (2,0) -- (0,0) -- cycle;
    
    \def\eps{0.2}
    \draw[domain=0:2, smooth, variable=\x,samples=500] plot ({\x}, {atan(1/cos(\x*45))/90*2});
    \node at (1.3,0.8) {$\tan \psi \leq \sec \phi$};
    \draw (2,2)--(2,0) node[below] {$\frac{\pi}{2}$};
    \node[left] at (0,1) {$\frac{\pi}{4}$};
    \draw (2,2)--(0,2) node[left] {$\frac{\pi}{2}$};
    \draw[dotted] (\eps, 2) -- (\eps,0) node[below] {$\varepsilon$};
    \draw[dotted] (\eps, 1.5) -- (0, 1.5) node[left] {$\frac{3\pi}{8}$};
    \draw[dotted] (\eps, \eps) -- (0, \eps) node[left] {$\varepsilon$};
    \node[below left] at (0,0) {$0$};
\end{tikzpicture}
\caption{The domain $[0,\frac{\pi}{2})^2$ with the two regions bounded by the condition $\tan\psi = \sec \phi$.} \label{fig:regions}
\end{figure}
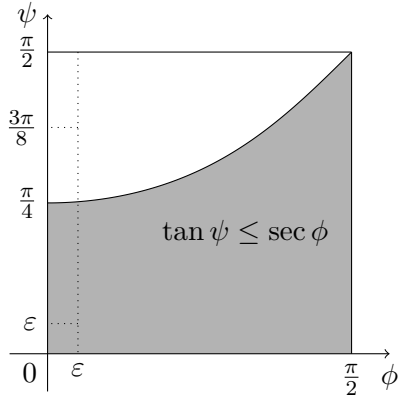

In order to proceed we need to derive estimates for $I_1(Z(\phi,\psi,\bu))$ and $I_2(Z(\phi,\psi,\bu))$ for $\phi\in (0,\pi)$, $\psi\in(0,\pi/2)$ and $\bu\in\mathbb{S}^{d-2}$. This is the purpose of Step 2 presented in the next section.

\section{Proof of Theorem \ref{thm:main}, Step 2: Estimates for \texorpdfstring{$I_1$}{I1} and \texorpdfstring{$I_2$}{I2}}\label{sec:ProofStep2}

\subsection{Estimate for \texorpdfstring{$I_1$}{I1}}

Let us start with the analysis of $I_1(Z(\phi,\psi,\bu))$. First of all we will prove, that $\mathbb{S}_{2,+}^d\cap H(\bz)$ is a spherical wedge.

\begin{lemma}[The Shape of $\SS_{2,+}^d\cap H(\bz)$]\label{lm:Shape}
The spherical convex set $\SS_{2,+}^d\cap H(\bz)$ is a spherical wedge of dimension $d-1$ with an opening angle $\beta(\bz)\in[0,\pi]$. Moreover for any $\phi\in[0,\pi]$,$\psi\in[0,\frac{\pi}{2}]$ and $\bu \in \SS^{d-2}$ we have for $\beta(Z(\phi,\psi,\bu))=\beta$, that
\begin{equation}\label{eqn:beta_exact}
    \sin\beta = \frac{\sin\phi}{\sqrt{(\cos\psi)^2+(\sin\phi)^2(\sin\psi)^2}} \quad\text{and}\quad
    \cos\beta = \frac{(\cos\phi)(\cos\psi)}{\sqrt{(\cos\psi)^2+(\sin\phi)^2(\sin\psi)^2}}
\end{equation}
and therefore
\begin{equation*}
    \tan\beta=\frac{\tan\phi}{\cos\psi}.
\end{equation*}
Furthermore, if $\varepsilon\in(0,\frac{1}{2})$ then
\begin{equation}\label{eqn:beta_asy}
    \tan \beta = (1+O(\varepsilon)) \phi \qquad \text{for all $\phi,\psi\in(0,\varepsilon)$.}
\end{equation}
\end{lemma}

\begin{proof}
In order to prove this statement  we may consider the positive hull of $\SS_{2,+}^d\cap H(\bz)$ in $\RR^{d+1}$ which yields a $d$-dimensional cone $C$ with a lineality space $L=H(\be_{d+1})\cap H(\be_d)\cap H(\bz)$ of dimension $d-2$ if $\bz$ is in general position and in this case $\mathrm{proj}_{L^\bot} C = L^\bot \cap C$ is a proper $2$-dimensional cone with opening angle $\beta:=\beta(\bz)$.
To prove \eqref{eqn:beta_exact}, we choose the parametrization $\bz =Z(\phi,\psi,\bu)$ as introduced in Section \ref{sec:ProofStep1}.
We first note that $L=H(\be_{d+1})\cap H(\be_d)\cap H(\bz) = H(\be_{d+1})\cap H(\be_d)\cap H(\bu)$,
\begin{equation*}
    H(\bz)\cap \SS_{2,+}^d = \{\by\in \SS^d: \by\cdot \be_{d+1}\geq 0, \by\cdot \be_d\geq 0, \by\cdot \bz=0\}
\end{equation*}
and 
\begin{equation*}
    \by\cdot \bz = 0\quad \Longleftrightarrow \quad (\sin \phi) (\sin \psi) (\by\cdot \bu) = (\cos \psi) (\by\cdot \be_{d+1})+(\cos\phi)(\sin \psi) (\by\cdot \be_d).
\end{equation*}
Thus, the $2$-dimensional wedge $\mathrm{proj}_{L^\bot}C\subset L^\bot = \operatorname{span}\{\be_{d+1},\be_d,\bu\}$ is spanned by the two vectors
\begin{equation*}
    \by_1 = (\sin\phi)(\sin\psi)\be_{d+1} + (\cos \psi)\bu \qquad \text{ and } \qquad \by_2=(\sin\phi)\be_d+(\cos \phi)\bu,
\end{equation*}
see Figure \ref{fig:2dwedge}.
\begin{figure}
    \centering
    \begin{tikzpicture}
        \draw[->] (0,0) -- (100:2) node[above] {$\be_{d+1}$};
        \draw[->] (0,0) -- (-30:2) node[right] {$\be_d$};
        \draw[->] (40:1.5) -- (-140:1.5) node[below]{$\bu$};
        
        \fill[gray, opacity=0.3] (-140:1.3) -- (40:1.3) --++ (100:1.8) --++ (-140:2.6)--cycle;
        
        \fill[gray, opacity=0.3] (-140:1.3) -- (40:1.3) --++ (-30:1.8) --++ (-140:2.6)--cycle;
        
        \fill[orange,opacity=0.3] (0,0) -- (0.95,1.9) -- (0.85,-1.7) -- cycle;
        
        \draw[thick] (0,0) -- (0.95,1.9) node[above right] {$\by_1$};
        \draw[thick] (0,0) -- (0.85,-1.7) node[below right]{$\by_2$};
        
        \draw[-] (0.2,-0.4) arc (-30:35:0.7) node[below right] {$\beta$};
        
        \node at (1.5,0.5) {$\mathrm{proj}_{L^\bot} C$};
    \end{tikzpicture}
    \caption{Illustration of the $2$-dimensional wedge $\mathrm{proj}_{L^\bot} C$.}
    \label{fig:2dwedge}
\end{figure}
Thus,
\begin{align*}
    \sin \beta &= \frac{\|\by_1\times \by_2\|}{\|\by_1\|\|\by_2\|} 
    = \frac{\sin \phi}{\sqrt{(\cos\psi)^2 + (\sin\phi)^2(\sin\psi)^2}}\\
\intertext{and}
    \cos \beta &= \frac{\by_1\cdot \by_2}{\|\by_1\|\|\by_2\|}
    = \frac{(\cos \phi)(\cos\psi)}{\sqrt{(\cos\psi)^2 + (\sin\phi)^2(\sin\psi)^2}}
\end{align*}
for all $\phi\in[0,\pi]$ and $\psi\in[0,\frac{\pi}{2}]$, proving \eqref{eqn:beta_exact}.

Now, if $\varepsilon \in (0,\frac{1}{2})$, then
\begin{align*}
    \phi \leq \tan\beta = \frac{\tan\phi}{\cos\psi} \leq \frac{1+\varepsilon}{1-\varepsilon} \phi\leq (1+4\varepsilon)\phi \qquad \text{for all } \phi,\psi\in (0,\varepsilon).
\end{align*}
Thus, $\tan \beta = (1+O(\varepsilon)) \phi$, which proves \eqref{eqn:beta_asy}.
\end{proof}

We are now ready to formulate the result for $I_1(Z(\phi,\psi,\bu))$.

\begin{lemma}\label{lm:I_1}
For any $\phi\in[0,\pi)$, $\psi\in[0,\frac{\pi}{2})$ and $\bu\in\SS^{d-2}$ we have that
\begin{equation*}
    I_1(Z(\phi,\psi,\bu)) \leq \left(\frac{\omega_{d}}{2}\right)^{d}.
\end{equation*}
Furthermore, for any $\phi\in[0,\pi)$, $\psi\in[0,\frac{\pi}{2})$ and $\bu\in\SS^{d-2}$ we have
\begin{align*}
    I_1(Z(\phi,\psi,\bu)) 
    \leq \frac{A_d}{2} \left(\frac{\omega_{d+1}}{4\pi}\right)^d \left(2\tan\frac{\beta(Z(\phi,\psi,\bu))}{2}\right)^{d+1},
\end{align*}
and, if $\varepsilon\in(0,\frac{1}{2})$, then
\begin{align*}
    I_1(Z(\phi,\psi,\bu)) = (1+O_d(\varepsilon)) \frac{A_d}{2} \left(\frac{\omega_{d+1}}{4\pi}\right)^d \phi^{d+1}, \qquad \forall \phi,\psi\in(0,\varepsilon).
\end{align*}
Here, the constant $A_d$ is given by
\begin{equation}\label{eqn:A_d}
    A_d := \EE\nabla_d((U_1,\bZ_1,1),\dotsc,(U_d,\bZ_d,1))
\end{equation}
with random variables $U_1,\ldots,U_d$ uniformly distributed on $[-1,1]$ and random vectors $\bZ_1,\dotsc,\bZ_d$ distributed according to a beta-prime distribution on $\RR^{d-2}$ with parameter $\frac{d+1}{2}$ and probability density function $\widetilde{f}_{d-2,\frac{d+1}{2}}$ as in \eqref{eq:betaprimedensity} in such a way that $U_1,\ldots,U_d,\bZ_1,\ldots,\bZ_d$ are independent. 
\end{lemma}
\begin{proof}
As it was shown in Lemma \ref{lm:Shape}, $W(\phi,\psi,\bu):=H(Z(\phi,\psi,\bu))\cap \SS_{2,+}^d$ 
is a $(d-1)$-dimensional spherical wedge with an opening angle $\beta:=\beta(Z(\phi,\psi,\bu))$. We recall, that
\begin{equation*}
    I_1:= I_1(Z(\phi,\psi,\bu))
    =\int_{W(\phi,\psi,\bu)}\cdots\int_{W(\phi,\psi,\bu)}\nabla_{d}(\bx_1,\ldots,\bx_d)\,\sigma_{d-1}(\dint \bx_1)\ldots\sigma_{d-1}(\dint \bx_d).
\end{equation*}
Since the volume of the parallelotope $W_d^0=\{\sum_{i=1}^d\lambda_i\bx_i: (\lambda_1,\dotsc,\lambda_d)\in[0,1]^d\}$ is bounded by the product of its side lengths we conclude that $\nabla_d(\bx_1,\ldots,\bx_d)=\mathrm{vol}_d(W_d^0)\leq \prod_{i=1}^d\|\bx_i\|= 1$, which yields
\begin{equation*}
    I_1 \leq \sigma_{d-1}(W(\phi,\psi,\bu))^d \leq (\omega_{d}/2)^d.
\end{equation*}
This settles the first inequality. 

For the next step we consider the gnomonic projection $g^{d-1}_{\bp}$ with respect to the center of symmetry $\bp\in\SS^d$ of the spherical wedge $W(\phi,\psi,\bu)$. 
After applying a corresponding rotation we obtain, 
that the image of the set $W(\phi,\psi,\bu)$ under the gnomonic projection $g^{d-1}_{\bp}$ is 
$[-\tan \frac{\beta}{2},\tan \frac{\beta}{2}]\times \RR^{d-2}$, 
which is in particular independent of $\bu$ since $\beta$ is independent of $\bu$ (see Lemma \ref{lm:Shape}). 
Thus, by definition of the inverse gnomonic projection \eqref{eq:gnomonicprojinv} 
and by \cite[Proposition 4.2]{BW16} we have
\begin{equation}\label{eq:Eq1}
    \begin{aligned}
        I_1
        &= \int_{\left([-\tan \frac{\beta}{2},\tan\frac{\beta}{2}]\times\RR^{d-2}\right)^d} \nabla_d\left(\frac{\by_1+\bp}{\sqrt{1+\|\by_1\|^2}}, \dotsc, \frac{\by_d+\bp}{\sqrt{1+\|\by_d\|^2}}\right) \prod_{i=1}^d \frac{\dint\by_i}{(1+\|\by_i\|^2)^{\frac{d}{2}}}
       \\
       &=\int_{\left([-\tan \frac{\beta}{2},\tan\frac{\beta}{2}]\times\RR^{d-2}\right)^d} \nabla_d(\by_1+\bp, \dotsc, \by_d+\bp) \prod_{i=1}^d \frac{\dint\by_i}{(1+\|\by_i\|^2)^{\frac{d+1}{2}}}.
    \end{aligned}
\end{equation}
Further applying the change of variables 
\begin{equation*}
    \by_i+\bp=\Big(\tan(\beta/2)\,u_i,\bz_i,1\Big),\qquad 1\leq i\leq d,
\end{equation*}
and recalling \eqref{eq:betaprimedensity} we obtain
\begin{align*}
    I_1
        &= \left(\tan \frac{\beta}{2}\right)^{d+1}\int_{([-1,1]\times \RR^{d-2})^d} 
            \nabla_d\big((u_1,\bz_1,1),\dotsc,(u_d,\bz_d,1)\big) 
            \prod_{i=1}^d \frac{\dint u_i\, \dint\bz_i}{(1+(\tan \frac{\beta}{2})^2 u_i^2+\|\bz_i\|^2)^{\frac{d+1}{2}}}\\
        &\leq \left(\tan \frac{\beta}{2}\right)^{d+1} 2^d\, \tilde{c}_{d-2,\frac{d+1}{2}}^{-d}
            \int_{([-1,1]\times \RR^{d-2})^d} 
                \nabla_d\big((u_1,\bz_1,1),\dotsc,(u_d,\bz_d,1)\big)
                \prod_{i=1}^d \frac{\dint u_i}{2} \, \widetilde{f}_{d-2,\frac{d+1}{2}}(\bz_i)\dint\bz_i\\
        &= \frac{A_d}{2} \left(\frac{\omega_{d+1}}{4\pi}\right)^d \left(2\tan\frac{\beta}{2}\right)^{d+1},
\end{align*}
where we used the fact that
\begin{align*}
    \tilde{c}_{d-2,\frac{d+1}{2}} 
        = \frac{\Gamma(\frac{d+1}{2})}{\pi^{\frac{d-2}{2}} \Gamma\left(\frac{d+1}{2} - \frac{d-2}{2}\right)} 
        = \frac{2\Gamma\left(\frac{d+1}{2}\right)}{\pi^{\frac{d-1}{2}}} = \frac{4\pi}{\omega_{d+1}}.
\end{align*}

In particular, for $\varepsilon\in(0,\frac{1}{2})$ this yields, by Lemma \ref{lm:Shape} and the fact that $2\tan(\beta/2)\leq \tan \beta$ for all $\beta\in[0,\pi/2]$, that
\begin{align*}
    I_1 
        \leq \frac{A_d}{2} \left(\frac{\omega_{d+1}}{4\pi}\right)^d \left(\frac{\tan \phi}{\cos\psi}\right)^{d+1}
        \leq \frac{A_d}{2} \left(\frac{\omega_{d+1}}{4\pi}\right)^d ((1+4\varepsilon)\phi)^{d+1}
        = (1+O_d(\varepsilon)) \frac{A_d}{2} \left(\frac{\omega_{d+1}}{4\pi}\right)^d \phi^{d+1},
\end{align*}
for all $\phi,\psi \in (0,\varepsilon)$.

To conclude the lower bound we just note that by Lemma \ref{lm:Shape},
\begin{equation*}
    \frac{1}{(1+(\tan \frac{\beta}{2})^2 u_i^2 + \|\bz_i\|^2)^{\frac{d+1}{2}}} 
    \geq\frac{1}{(1 + \|\bz_i\|^2 + (\tan \frac{\beta}{2})^2)^{\frac{d+1}{2}}}
    \geq \frac{1-O_d(\varepsilon)}{(1+\|\bz_i\|^2)^{\frac{d+1}{2}}},
\end{equation*}
which yields, by \eqref{eqn:beta_asy}, that
\begin{align*}
    I_1 
    &\geq (1-O_d(\varepsilon)) \frac{A_d}{2} \left(\frac{\omega_{d+1}}{4\pi}\right)^d \left(2\tan \frac{\beta}{2}\right)^{d+1}= (1-O_d(\varepsilon)) \frac{A_d}{2} \left(\frac{\omega_{d+1}}{4\pi}\right)^d \phi^{d+1},
\end{align*}
for all $\phi,\psi\in(0,\varepsilon)$.
\end{proof}

\begin{remark}\label{rem:A2}
    We were not able to determine explicitly the value of the constant $A_d$ in \eqref{eqn:A_d} for $d>2$. However, for $d=2$ we have that
    \begin{equation*}
        A_2 = \int_{-1}^1\int_{-1}^1 |x-y| \, \frac{dx}{2}\, \frac{dy}{2} = \frac{2}{3}.
    \end{equation*}
\end{remark}

\subsection{Estimate for \texorpdfstring{$I_2$}{I2}}

We continue with investigating $I_2(Z(\phi,\psi,\bu))$.

\begin{lemma}\label{lm:I_2}
For any $\phi\in[0,\pi)$, $\psi\in\big[0,\frac{\pi}{2}\big)$ and $\bu\in\SS^{d-2}$ we have
\begin{align*}
    I_2(Z(\phi,\psi,\bu)) &= \frac{\omega_{d+1}}{ 4\pi}(\psi-\arcsin(\cos\phi\,\sin\psi)).
\end{align*}
Moreover, there exists $\varepsilon_0>0$ such that for all $\varepsilon\in(0,\varepsilon_0)$ we have
\begin{equation}\label{eqn:I2_asy}
    I_2(Z(\phi,\psi,\bu)) = (1+O(\varepsilon)) \frac{\omega_{d+1}}{8\pi} \phi^2\psi, \qquad \text{for all } \phi,\psi\in(0,\varepsilon),
\end{equation}
and
\begin{align}\label{eqn:I2_bound}
    I_2(Z(\phi,\psi,\bu)) &\geq \frac{\omega_{d+1}}{2\pi^3} \phi^2\psi, \qquad \text{for all } \phi \in (0,\pi), \psi\in (0,\pi/2).
\end{align}
\end{lemma}

\begin{proof}
We will start with the case $d=2$. In this situation $\mathbb{S}^d\cap H(\be_{3})\cap H(\be_2)=\{\pm \be_1\}$ and, thus,
\begin{equation*}
    Z(\phi,\psi):=Z(\phi,\psi,\be_1)=(\sin\phi)(\sin\psi) \be_1 - (\cos\phi)(\sin\psi) \be_2 - (\cos \psi) \be_{3},
\end{equation*}
which describes the usual spherical coordinates in $\RR^3$. Further,
\begin{equation*}
    I_2(Z(\phi,\psi))=\sigma_2(\SS_{2,+}^{2}\cap H^{+}(Z(\phi,\psi))),
\end{equation*}
where $\SS_{2,+}^{2}\cap H^+(Z(\phi,\psi))$ is a spherical triangle with angles $\frac{\pi}{2}$, $\psi$ and $\arccos(\cos\phi\,\sin\psi)$. Then by Girard's theorem about the area of a spherical triangle we conclude
\begin{equation}\label{eqn:vol_s2+_dim2}
    \sigma_2(\SS_{2,+}^{d}\cap H^+(Z(\phi,\psi)))
    =\psi+\arccos(\cos\phi\,\sin\psi)+\frac{\pi}{2}-\pi=\psi -\arcsin(\cos\phi\,\sin\psi).
\end{equation}

Let us now consider the case $d\ge 3$. First of all we note that for $\bz\in\SS^d$ given by
\begin{equation*}
    I_2(\bz)
    =\sigma_d(\mathbb{S}^d_{2,+}\cap H^{+}(\bz))
    =\int_{\mathbb{S}^d}{\bf 1}(\bx\in \mathbb{S}^d_{2,+}\cap H^{+}(\bz)) \, \sigma_d(\dint\bx).
\end{equation*}
Due to invariance of the spherical Lebesgue measure with respect to rotations and for convenience, we apply a rotation $\rho_d$ to the wedge $\mathbb{S}_{2,+}^d$ so that the normal vector to the hyperplane $H_1$ is $\frac{1}{\sqrt{2}}(\be_{d+1}-\be_d)$  and the normal vector to the hyperplane $H_2$ is $\frac{1}{\sqrt{2}}(\be_{d+1}+\be_d)$. We keep the notation $\mathbb{S}^d_{+}=\{\bx\in\mathbb{S}^d\colon x_{d+1}\ge 0\}$ and let $\bz'=\rho_d(\bz)$. Consider a gnomonic projection $g^d_{\bp_d}$ with respect to the north pole $\bp_d:=(0,0,\ldots, 1)$ of the sphere $\SS^d$
\begin{align*}
    g^d_{\bp_d}: \left\{ \begin{array}{ccc} \SS^d_+&\to& \RR^d,\\
    (x_1,\ldots,x_d,x_{d+1})&\mapsto& \Big(\frac{x_1}{x_{d+1}},\ldots, \frac{x_d}{x_{d+1}}\Big),\end{array}\right.
\end{align*}
 where we identify $\lspan\{\be_1,\ldots,\be_d\}$ with $\RR^d$. Under $g^d_{\bp_d}$ the half great hyperspheres $H_1\cap \mathbb{S}^d_{+}$ and $H_2\cap \mathbb{S}^d_{+}$ are mapped onto the hyperplanes $L_1(\bu_1), L_1(\bu_2) \subset \RR^d$ with the normal vectors $\bu_1=(0,\ldots,0,1)$ and $\bu_2=(0,\ldots,0,-1)$, respectively, and the distance one to the origin.
Analogously, the half great hypersphere $H(\bz')\cap \mathbb{S}^d_{+}$ is mapped onto the hyperplane $L(\bw)$ with normal vector 
\begin{equation*}
    \bw = \left(-\frac{z'_1}{\sqrt{1-(z'_{d+1})^2}}, \ldots, -\frac{z'_d}{\sqrt{1-(z'_{d+1})^2}}\right)
\end{equation*}
and the distance $\frac{z'_{d+1}}{\sqrt{1-(z'_{d+1})^2}}$ to the origin. Thus, by \cite[Proposition 4.2]{BW16} we have that
\begin{equation*}
    \frac{\sigma_d(\mathbb{S}^d_{2,+}\cap H^{+}(\bz))}{\sigma_d(\mathbb{S}_+^d)}
    = \tilde{c}_{d,\frac{d+1}{2}}\int_{\RR^d}{\bf 1}(\by\in L_1(\bu_1)^{-}\cap L_1(\bu_2)^{-}\cap L(\bw)^{+})(1+\|\by\|^2)^{-\frac{d+1}{2}}\dint\by.
\end{equation*}
Next, we use the fact that the beta-prime density $\widetilde{f}_{d,\frac{d+1}{2}}$ is invariant with respect to rotations and we consider a rotation $\theta: \RR^d\mapsto\RR^d$, such that
\begin{equation*}
    \theta \bu_1 
    = \bu_1,\qquad \theta \bu_2
    = \bu_2, \qquad \theta \bw 
    = (0,\ldots,0,t_{d-1},t_d).
\end{equation*}
Since $\theta$ leaves the $d$th coordinate axis and the hyperplanes $L_1(\bu_1),L_1(\bu_2)$ unchanged, we have that
\begin{equation*}
    w_d=(\theta\bw)_d =t_d= -z'_d/\sqrt{1-(z'_{d+1})^2}.
\end{equation*}
On the other hand $\|\bw\|=\|\theta\bw\|=1$, and so 
\begin{equation*}
    (\theta\bw)_{d-1} 
    = t_{d-1} 
    = -\sqrt{\frac{1-(z'_{d+1})^2-(z'_d)^2}{1-(z'_{d+1})^2}}.
\end{equation*}
Thus, applying the change of variables $\by=\theta^{-1}\tilde{\by}$ we arrive at
\begin{equation*}
    \frac{\sigma_d(\mathbb{S}^d_{2,+}\cap H^{+}(\bz))}{\sigma_d(\mathbb{S}_+^d)}
    = \tilde{c}_{d,\frac{d+1}{2}}\int_{\RR^d}{\bf 1}(\tilde{\by}\in L_1(\bu_1)^{-}\cap L_1(\bu_2)^{-}\cap L(\theta\bw)^{+})(1+\|\tilde{\by}\|^2)^{-\frac{d+1}{2}}\dint\tilde{\by}.
\end{equation*}
Now, we use another crucial property of the beta-prime distribution, which says that under orthogonal projection $\proj_M:\RR^d\mapsto M$ onto a $k$-dimensional linear subspace $M$ the beta-prime distribution on $\RR^d$ with density $\widetilde{f}_{d,\beta}$ maps to a beta-prime distribution on $M$ with density $\widetilde{f}_{k, \beta-\frac{d-k}{2}}$. Strictly speaking, if a random vector $X$ on $\RR^d$ has density function $\widetilde{f}_{d,\beta}$, then the projected random vector $I_M(\proj_M(X))$ has density $\widetilde{f}_{k, \beta-\frac{d-k}{2}}$, where $I_M: M\mapsto \lspan\{\be_1,\ldots, \be_k\}$ is an isometry such that $I_M(\proj_M(0))=0$ (see \cite[Lemma 3.1]{KTZ}). We consider the projection $\proj_E$ onto the $2$-dimensional linear subspace $E:=\{\by\in\RR^d\colon y_1=\ldots=y_{d-2}=0\}$, which we identify with $\RR^2$. Under this projection we have
\begin{align*}
    \proj_E(L_1(\bu_1)) &= \{(s_{d-1},s_d)\in\RR^2\colon s_d=1\},\\ 
    \proj_E(L_1(\bu_2))  &= \{(s_{d-1},s_d)\in\RR^2\colon s_d=-1\},\\
    \proj_E(L(\theta\bw)) &= \left\{(s_{d-1},s_d)\in\RR^2\colon -s_{d-1}\sqrt{1-(z'_{d+1})^2-(z'_{d})^2}-s_dz'_d=z'_{d+1}\right\},
\end{align*}
and
\begin{align*}
    \frac{\sigma_d(\mathbb{S}^d_{2,+}\cap H^{+}(\bz))}{\sigma_d(\mathbb{S}_+^d)}
    &=\tilde{c}_{2,\frac{3}{2}}\int_{\RR^2}{\bf 1}((s_{d-1},s_d)\in \proj_E(L_1(\bu_1))^{-}\cap \proj_E(L_1(\bu_2))^{-}\cap \pi_E(L(\theta\bw))^{+})\\
        &\qquad\qquad\qquad\qquad\times(1+s_{d-1}^2+s_d^2)^{-\frac{3}{2}}\dint s_{d-1}\dint s_d.
\end{align*}
Considering the gnomonic projection $g^2_{\bp_2}$ of the two-dimensional half-sphere $\mathbb{S}^2_{+}$ with respect to the north pole $\bp_2:=(0,0,1)$ we note that the integral above describes the spherical area of the spherical triangle $\mathbb{S}^2\cap \widetilde{H}_1^+\cap \widetilde{H}_2^{+}\cap H(\bv')^+$, where $\widetilde{H}_1$ and $\widetilde{H}_2$ are $2$-dimensional linear subspaces in $\RR^3$ with normal vectors $\frac{1}{\sqrt{2}}(\be_3-\be_2)$ and $\frac{1}{\sqrt{2}}(\be_3+\be_2)$, respectively, and $\bv'=(\sqrt{1-(z'_{d})^2-(z'_{d+1})^2},z'_d,z'_{d+1})$. Finally we consider a rotation $\rho_3$ in $\RR^3$ which is defined as an image of the restriction of rotation $\rho_d$ on the linear subspace $\lspan\{\be_{d+1},\be_d,\be_{d-1}\}$ under the isometry $I:\lspan\{\be_{d+1},\be_d,\be_{d-1}\}\mapsto \RR^3$ satisfying $I(\be_{d+1})=\be_3$ and $I(\be_d)=\be_2$. Since the last two coordinates of the vectors $\bz'$ and $\bv'$ are identical, the last two coordinates of the vectors $\bz=\rho_d^{-1}(\bz')$ and $\bv:=\rho_3^{-1}(\bv')$ are identical as well. From this we finally conclude that
\begin{equation*}
    \sigma_d(\mathbb{S}^d_{2,+}\cap H^{+}(\bz))
    ={\frac{\sigma_d(\mathbb{S}_+^d)}{\sigma_2(\mathbb{S}_+^2)}}
        \sigma_2(\mathbb{S}^2_{2,+}\cap H^{+}(\bv)),
\end{equation*}
which has already been computed at the beginning of the proof (see \eqref{eqn:vol_s2+_dim2}) and depends only on the last two coordinates $z_{d+1},z_{d}$ of the vector $\bz$. Thus, 
\begin{equation}\label{eq:I2-}
    I_2(Z(\phi,\psi,\bu))=\frac{\omega_{d+1}}{4\pi}(\psi-\arcsin(\cos\phi\,\sin\psi)).
\end{equation}

In order to prove the second statement of the lemma we consider the function
\begin{equation*}
    f(\phi,\psi):=\frac{\psi-\arcsin(\cos\phi\,\sin\psi)}{\psi\phi^2},\qquad \phi\in(0,\pi],\,\psi\in(0,\pi/2].
\end{equation*}
In Appendix \ref{sec:calc_f} we show that there exists $\varepsilon>0$ such that
\begin{align}
    \frac{1}{2}-\varepsilon \leq f(\phi,\psi) \leq \frac{1}{2} + \varepsilon,\qquad \text{for all } \phi,\psi\in(0,\varepsilon),
\end{align}
and that $f$ has the absolute minimum $2/\pi^2$, which is achieved along the edge $\phi = \pi$ in the rectangle $(\phi,\psi)\in [0,\pi]\times [0,\frac{\pi}{2}]$. This yields \eqref{eqn:I2_asy} and \eqref{eqn:I2_bound}.
\end{proof}

\section{Proof of Theorem \ref{thm:main}, Step 3: Conclusion for the Poisson model}\label{sec:ProofStep3}

We are now ready to finish the proof of Theorem \ref{thm:main} for the Poisson model $K_{\eta_\I}^{(s,2)}$. Before we continue let us summarize the bounds obtained in Lemmas \ref{lm:I_1} and \ref{lm:I_2} in the form we will apply them. There exist $0<\varepsilon_0<1/2$ and constants $c_d, C_d, \theta_d, \Theta_d>0$, such that for all $\bu\in\mathbb{S}^{d-2}$ and $\varepsilon \in(0,\varepsilon_0)$ we have that
\begin{align*}
    \phi,\psi\in\big(0,\varepsilon\big):\quad 
        B_d(1-\theta_d\varepsilon)\phi^{d+1}
            &\leq I_1(Z(\phi,\psi,\bu))
            \leq B_d(1+\Theta_d\varepsilon)\phi^{d+1},\\
        b_d (1-\theta_d\varepsilon) \phi^2\psi
            &\leq I_2(Z(\phi,\psi,\bu))
            \leq b_d(1+\Theta_d\varepsilon)\phi^2\psi,
\end{align*}
where
\begin{equation*}
    b_d = \frac{\omega_{d+1}}{8\pi} \qquad \text{and} \qquad B_d = \frac{A_d}{2} \left(\frac{\omega_{d+1}}{4\pi}\right)^d,
\end{equation*}
and
\begin{align*}
    \phi\in\big(0,\varepsilon\big), \psi\in\big[\varepsilon,3\pi/8\big):&
        &I_1(Z(\phi,\psi,\bu)) &\leq C_d \phi^{d+1},
        &I_2(Z(\phi,\psi,\bu))&\ge c_d\phi^2\varepsilon,\\
    \phi\in\big[\varepsilon,\pi\big), \psi\in\big(0,\pi/2\big) :&
        & I_1(Z(\phi,\psi,\bu))&\leq C_d,
        &I_2(Z(\phi,\psi,\bu)) &\ge c_d\varepsilon^2\psi.
\end{align*}
Note that we used Lemma \ref{lm:Shape} to derive that
\begin{equation*}
    2\tan \frac{\beta}{2} \leq \frac{2\sin \beta}{1+\cos \beta} \leq \frac{2\phi}{\cos(3\pi/8)} \leq 6\phi 
    \quad \text{for all $\phi\in(0,\varepsilon)$, $\psi\in[\varepsilon, 3\pi/8)$.}
\end{equation*}
Proceeding with Estimate \eqref{eqn:lower} we find that
\begin{align*}
    \mathbb{E}f_{d-1}(K_{\eta_\I}^{(s,2)}) 
    &\geq \frac{2}{d!} \I^d \int_{\SS^{d-2}} \int_{0}^{\varepsilon}\int_{0}^{\varepsilon} (\sin \phi)^{d-2} (\sin \psi)^{d-1} I_1(Z(\phi,\psi,\bu)) \\
   &\qquad\qquad\times \exp(-\I I_2(Z(\phi,\psi,\bu))) \, \dint \psi\, \dint \phi \, \sigma_{d-2}(\dint \bu)\\
   &\geq (1-\theta_d\varepsilon)  \frac{2\omega_{d-1} B_d }{d!}\I^d \int_{0}^{\varepsilon}\int_{0}^{\varepsilon} (\sin \phi)^{d-2} (\sin \psi)^{d-1} \phi^{d+1}e^{-\I b_d(1+\Theta_d\varepsilon)\phi^2\psi} \, \dint \psi\, \dint \phi.
\end{align*}
Using the bound $\sin x\ge (1-\varepsilon)x$, which is valid for $x\in\big[0,\varepsilon\big)$, we arrive at
\begin{align}
    \mathbb{E}f_{d-1}(K_{\eta_\I}^{(s,2)})
    &\geq (1-O_d(\varepsilon)) \frac{2\omega_{d-1} B_d}{d!} \I^d\int_{0}^{\varepsilon}\int_{0}^{\varepsilon} \phi^{2d-1} \psi^{d-1} e^{-\I b_d(1+\Theta_d\varepsilon)\phi^2\psi} \, \dint \psi\, \dint \phi \notag\\
    &=(1-O_d(\varepsilon)) \frac{\omega_{d-1} B_d}{d!} \I^d \int_{0}^{\varepsilon^2}\int_{0}^{\varepsilon} s^{d-1}t^{d-1} e^{-\I b_d(1+\Theta_d\varepsilon) st} \, \dint t\, \dint s \notag\\
    &= (1-O_d(\varepsilon)) \frac{ \omega_{d-1} B_d}{d! b_d^d} \int_{0}^{\I b_d(1+\Theta_d\varepsilon)\varepsilon^2} s^{d-1} \int_{0}^{\varepsilon} t^{d-1} e^{-st} \, \dint t\, \dint s.\label{eq:Eq3}
\end{align}

For the upper bound we first note that
\begin{equation}\label{eqn:trivial_lower}
\begin{aligned}
    \sigma_d(\mathbb{S}_{2,+}^d) - I_2(Z(\phi,\psi,\bu)) 
    &= \frac{\omega_{d+1}}{4\pi} \Big(\pi-\psi+\arcsin(\cos(\phi)\sin(\psi))\Big) \\
    &\geq \frac{\omega_{d+1}}{4\pi} \frac{1}{3}\left((\frac{\pi}{2}-\psi)+(\pi-\phi)\right)
    = \frac{\omega_{d+1}}{4\pi} \frac{1}{3}\left(\tilde{\psi}+\tilde{\phi}\right),
\end{aligned}
\end{equation}
for all $\phi\in(0,\pi)$ and $\psi\in(0,\pi/2)$ and with $\tilde{\phi}=\pi-\phi$ and $\tilde{\psi}=\pi/2-\psi$, see Appendix \ref{apx:proof_trivial_lower} for the 
proof of the inequality.
Using this on \eqref{eqn:upper} and applying the estimate $\sin x\leq x$, which holds for all $x\in\big[0,\pi\big]$, we find that
\begin{equation}\label{eq:Eq4}
\begin{aligned}
    \mathbb{E}f_{d-1}(K_{\eta_\I}^{(s,2)}) 
    &\leq (1+O_d(\varepsilon))\frac{2\omega_{d-1}B_d}{d!}\I^d\int_{0}^{\varepsilon}\int_{0}^{\varepsilon}
        \phi^{2d-1} \psi^{d-1}e^{-\I b_d (1-\theta_d\varepsilon)\phi^2\psi} \, \dint \psi\, \dint \phi\\
    &\qquad+ \frac{2\omega_{d-1}C_d}{d!}\I^d\int_{0}^{\varepsilon}\int_{\varepsilon}^{3\pi/8} 
        \phi^{2d-1} e^{-\I c_d\phi^2\varepsilon} \, \dint
     \psi\, \dint \phi\\
    &\qquad+ \frac{2\omega_{d-1}C_d}{d!} \I^d\int_{\varepsilon}^{\pi}\int_{0}^{\pi/2} \psi^{d-1}e^{-\I c_d\varepsilon^2\psi} \, \dint
     \psi\, \dint \phi \\
    &\qquad+ \frac{\omega_{d-1}C_d}{d!} \I^d\int_{0}^\pi\int_{0}^{\pi/2} \phi^{d-2} \exp\left(-\I \frac{\omega_{d+1}}{12\pi} (\phi+\psi)\right)\, \dint\psi\,\dint \phi \\
     &=:J_1+J_2+J_3+J_4.
\end{aligned}
\end{equation}
Let us consider each of the terms $J_1,J_2,J_3,J_4$ separately. For $J_1$ we get
\begin{align*}
    J_1&=(1+O_d(\varepsilon))\frac{\omega_{d-1}B_d}{d! b_d^d} \int_{0}^{\I b_d(1-\theta_d\varepsilon)\varepsilon^2}s^{d-1}
        \int_{0}^{\varepsilon} t^{d-1}e^{-st} \, \dint t\, \dint s.
\end{align*}
Next, for $J_2$ one has that
\begin{align*}
    J_2 &\leq \frac{3\pi \omega_{d-1} C_d}{4d!}   \int_{0}^{\varepsilon}\phi^{2d-1} \I^de^{-\I c_d\varepsilon\,\phi^2} \,\dint \phi\\
    &= \frac{3\pi \omega_{d-1} C_d}{4 d! \varepsilon^d c_d^d} \int_{0}^{\gamma c_d\varepsilon^3} s^{d-1} e^{-s}\, \dint s
    \leq \frac{3\pi \omega_{d-1} C_d}{4 d! \varepsilon^d c_d^d} \Gamma(d),
\end{align*}
and similarly
\begin{align*}
    J_3 & \leq \frac{2\pi \omega_{d-1}C_d}{d!} \int_{0}^{\pi/2}\psi^{d-1} \I^de^{-\I c_d\varepsilon^2\,\psi} \,\dint \psi\\ 
    &= \frac{2\pi \omega_{d-1}C_d}{d! \varepsilon^{2d}c^{d}_d}
        \int_{0}^{c_d\pi\varepsilon^2\I/2} t^{d-1}e^{-t} \dint t
    \leq \frac{2\pi \omega_{d-1}C_d}{d!\varepsilon^{2d}c^{d}_d} \Gamma(d).
\end{align*}
Finally, for $J_4$ we find that
\begin{align*}
    J_4 &= \frac{\omega_{d-1} C_d}{d!} \left(\frac{12\pi}{\omega_{d+1}}\right)^{\!\!d}
        \left(\int_{0}^{\I\frac{\omega_{d+1}}{12\pi}} s^{d-2} e^{-s}\,\dint s\right) 
        \left(\int_{0}^{\I\frac{\omega_{d+1}}{12\pi}} e^{-t}\, \dint t\right)\\ 
    &\leq \frac{\omega_{d-1} C_d}{d!} \left(\frac{12\pi}{\omega_{d+1}}\right)^{\!\!d} \Gamma(d-1).  
\end{align*}
Thus,
\begin{equation*}
    \limsup_{\I \to \infty} \frac{J_2+J_3+J_4}{\log\I} = 0,
\end{equation*}
and finally setting
\begin{equation*}
    G(\alpha,\I):=\int_{0}^{\alpha\I} s^{d-1} \int_{0}^{\varepsilon} t^{d-1} e^{-st} \, \dint t\, \dint s
\end{equation*}
we find that $\lim_{\I\to\infty} G(\alpha,\I)=\infty$ for any $\alpha>0$. Using L'Hospital's rule we conclude that
\begin{align*}
   \lim_{\I\to \infty} 
        \frac{G(\alpha,\I)}{\log \I}
        = \lim_{\I\to\infty} \frac{\alpha \left(\alpha\I\right)^{d-1} \int_{0}^{\varepsilon} t^{d-1} e^{-\alpha\I \,t}\, dt}{1/\I}
        = \lim_{\I\to \infty} \int_{0}^{\alpha\I\varepsilon} z^{d-1} e^{-z} \, \dint z 
        = (d-1)!
\end{align*}
Combining all obtained bounds together with \eqref{eq:Eq3} and \eqref{eq:Eq4} yields
\begin{align*}
    \liminf_{\I\to \infty} \frac{\mathbb{E}f_{d-1}(K_{\eta_\I}^{(s,2)})}{\log\I} 
    &\geq (1-O_d(\varepsilon)) \omega_{d-1}\frac{B_d}{db_d^d},\\
    \intertext{and}
    \limsup_{\I\to\infty} \frac{\mathbb{E}f_{d-1}(K_{\eta_\I}^{(s,2)})}{\log\I} 
    &\leq (1+O_d(\varepsilon)) \omega_{d-1} \frac{B_d}{db_d^d}.
\end{align*}
Since this is true for arbitrarily small $\varepsilon\in (0,\varepsilon_0)$ we finally derive that
\begin{equation*}
    \mathbb{E}f_{d-1}(K_{\eta_\I}^{(s,2)}) = \frac{2^{d-1} A_d \omega_{d-1}}{d} \, (\log \I) (1+o_{d}(1)), \qquad \text{as $\I\to\infty$}.
\end{equation*}
This concludes the proof of Theorem \ref{thm:main} for the Poisson model.

\section{Proof of Theorem \ref{thm:main}, Step 4: Conclusion for the binomial model}\label{sec:ProofStep4}

Let us now consider the binomial model $K_n^{(s,2)}$. We have that
\begin{align*}
    \EE f_{d-1} (K_n^{(s,2)}) 
    &= \frac{1}{d!} \EE \sum_{1\leq i_1<\dotsc<i_d\leq n} 
        \mathbf{1}\{\bx_{i_1},\dotsc,\bx_{i_d} \text{ generate a facet of }K_n^{(s,2)}\}\\
    &= \binom{n}{d} \int_{\mathbb{S}_{2,+}^d} \cdots \int_{\mathbb{S}_{2,+}^d} 
        \PP(\bx_1,\dotsc,\bx_d\text{  generate a facet of }K_n^{(s,2)}) \,
        \frac{\sigma_d(\dint\bx_1)}{\sigma_d(\mathbb{S}_{2,+}^d)}\ldots\frac{\sigma_d(\dint\bx_d)}{\sigma_d(\mathbb{S}_{2,+}^d)},
\end{align*}
and by applying the spherical Blaschke-Petkantschin formula as in in Step 1 in Section \ref{sec:ProofStep1}, we see that
\begin{align*}
    \EE f_{d-1} (K_n^{(s,2)})
    &= \frac{\omega_{d+1}}{2\sigma_d(\mathbb{S}_{2,+}^d)^d} \binom{n}{d} \int_{G(d+1,d)} \\
    &\qquad \times \left[
        \int_{\mathbb{S}_{2,+}^d\cap H} \cdots \int_{\mathbb{S}_{2,+}^d\cap H}
        \nabla_d(\bx_1,\dotsc,\bx_d) \, \sigma_{d-1}(\dint\bx_1)\dotsc \sigma_{d-1}(\dint\bx_d)\right]\\
    &\qquad \times \left[\left( \frac{\sigma_{d}(\mathbb{S}_{2,+}^d\cap H^+)}{\sigma_{d}(\mathbb{S}_{2,+}^d)}\right)^{n-d} + 
        \left( \frac{\sigma_{d}(\mathbb{S}_{2,+}^d\cap H^-)}{\sigma_{d}(\mathbb{S}_{2,+}^d)}\right)^{n-d}\right] \,\nu_d(\dint H)\\
    &= \frac{1}{\sigma_d(\mathbb{S}_{2,+}^d)^d} \binom{n}{d} \int_{\mathbb{S}^d}
        I_1(\bz) \left(1-\frac{I_2(\bz)}{\sigma_d(\mathbb{S}_{2,+}^d)}\right)^{n-d} \, \sigma_d(\dint\bz).
\end{align*}
Using the same parametrization as before we conclude further that
\begin{align*}
    \EE f_{d-1} (K_n^{(s,2)})
    &= \frac{1}{\sigma_d(\mathbb{S}_{2,+}^d)^d} \binom{n}{d} 
        \int_{\mathbb{S}^{d-2}}\int_{0}^\pi \int_{0}^{\pi/2} I_1(Z(\phi,\psi,\bu))\\
    &\qquad \times  \left[\left(1-\frac{I_2(Z(\phi,\psi,\bu))}{\sigma_d(\mathbb{S}_{2,+}^d)}\right)^{n-d} + \left(\frac{I_2(Z(\phi,\psi,\bu))}{\sigma_d(\mathbb{S}_{2,+}^d)}\right)^{n-d}\right]\\
    &\qquad \times (\sin \phi)^{d-2} (\sin\psi)^{d-1} \, \dint \psi\, \dint \phi\, \sigma_{d-2}(\dint\bu).
\end{align*}
With the same bounds on $I_1$ and $I_2$ as developed in Step 2 in Section \ref{sec:ProofStep2} we obtain
\begin{align*}
    \EE f_{d-1}(K_n^{(s,2)}) 
    &\geq (1-O_d(\varepsilon)) \omega_{d-1} B_d \frac{2}{\sigma_d(\mathbb{S}_{2,+}^d)^d} \binom{n}{d}
        \int_{0}^\varepsilon \int_{0}^\varepsilon\!\!\phi^{2d-1} \psi^{d-1}\!\! \left(1- \frac{(1+\Theta\varepsilon)b_d}{\sigma_d(\mathbb{S}_{2,+}^d)} \phi^2\psi\right)^{n-d} 
            \!\!\!\!\!\, \dint \psi\, \dint\phi \\
    &\geq (1-O_d(\varepsilon)) \omega_{d-1} \frac{B_d}{d! b_d^d} \frac{d!}{n^d}\binom{n}{d} 
        \int_{0}^{n\frac{(1+\Theta\varepsilon)b_d}{\sigma_d(\mathbb{S}_{2,+}^d)}\varepsilon^2} 
            \int_{0}^\varepsilon s^{d-1} t^{d-1} \left(1-\frac{st}{n}\right)^{n-d} \, \dint t\, \dint s.
\end{align*}
We note that
\begin{equation*}
    \lim_{n \to \infty} \frac{d!}{n^d}\binom{n}{d} = 1,
\end{equation*}
and define
\begin{equation*}
    H(\alpha,n) 
    := \int_{0}^{n\alpha} \int_{0}^\varepsilon 
        s^{d-1} t^{d-1} \left(1-\frac{st}{n}\right)^{n-d}\, \dint t\,\dint s.
\end{equation*}
Then
\begin{equation*}
    \lim_{n\to \infty} \frac{H(\alpha,n)}{\log n} = (d-1)!,
\end{equation*}
see, for example, \cite[Lem.\ on p.\ 296]{AW91}.
Thus,
\begin{equation*}
    \liminf_{n\to \infty} \frac{\EE f_{d-1}(K_n^{(s,2)})}{\log n} 
    \geq (1-O_d(\varepsilon)) \omega_{d-1} \frac{B_d}{d b_d^d}.
\end{equation*}
The upper bound 
\begin{equation*}
    \limsup_{n\to\infty} \frac{\EE f_{d-1}(K_n^{(s,2)})}{\log n} 
    \leq (1+O_d(\varepsilon)) \omega_{d-1} \frac{B_d}{d b_d^d},
\end{equation*}
can be obtained in a similar fashion. Finally we conclude that
\begin{equation*}
    \EE f_{d-1}(K_n^{(s,2)}) = \frac{2^{d-1}A_d \omega_{d-1}}{d}\, (\log n) (1+o_d(1)),\qquad \text{as $n\to\infty$}.
\end{equation*}
This finishes the proof of Theorem \ref{thm:main}.\qed

\appendix

\section{Analytic Estimates}\label{sec:calc_f}

\subsection{Estimate for Step 2}

We aim to show that there exists $\varepsilon>0$ such that
\begin{align*}
    \frac{1}{2} - \varepsilon \leq \frac{x-\arcsin(\sin(x)\cos(y))}{xy^2}\leq \frac{1}{2} + \varepsilon \qquad \text{for all } x,y\in(0,\varepsilon).
\end{align*}
To prove that this is indeed the case we first recall that for $z\in(-1,1)$ we have
\begin{align*}
    \arcsin z 
    &= \sum_{k=0}^\infty \frac{(2k-1)!!}{(2k)!!} \frac{z^{2k+1}}{2k+1} 
    = z +\frac{1}{6}z^3 + \frac{3}{40} z^5 + \frac{5}{112} z^7 + ...\\
\intertext{and for $x,y\in\mathbb{R}$ we find}
    \sin(x)\cos(y) 
    &= \frac{\sin(x+y)+\sin(x-y)}{2} 
    = \frac{1}{2} \sum_{k=0}^\infty \frac{(-1)^k}{(2k+1)!} 
        \sum_{m=0}^{2k+1} \binom{2k+1}{m} x^{2k+1-m}y^{m} (1+(-1)^m)\\
    &= \sum_{k=0}^\infty \frac{(-1)^k}{(2k+1)!} \sum_{m=0}^k \binom{2k+1}{2m} x^{2(k-m)+1} y^{2m}\\
    &= \sin(x) - xy^2 \sum_{k=0}^\infty \frac{(-1)^k}{(2k+3)!} \sum_{m=0}^{k} \binom{2k+3}{2m+2} x^{2(k-m)} y^{2m}\\
    &= \sin(x) - \frac{xy^2}{2} (1+G(x,y)),
\end{align*}
where
\begin{equation*}
    G(x,y) 
    = -2\sum_{k=0}^\infty \frac{(-1)^k}{(2k+5)!} \sum_{m=0}^{k+1} \binom{2k+5}{2m+2} x^{2(k+1-m)} y^{2m}
    =-\frac{2x^2+y^2}{12}+o(\|(x,y)\|^2) = o(\|(x,y)\|),
\end{equation*}
as $(x,y)\to 0$.
Thus,
\begin{align*}
    &\arcsin(\sin(x)\cos(y))\\ 
    &\quad= \sum_{k=0}^\infty \frac{(2k-1)!!}{(2k)!!} \frac{1}{2k+1} 
        \sum_{m=0}^{2k+1} \binom{2k+1}{m} \sin(x)^{2k+1-m} \left(-\frac{xy^2}{2} (1+G)\right)^{m}\\
    &\quad= \arcsin(\sin(x)) + \sum_{k=0}^\infty \frac{(2k-1)!!}{(2k)!!} \frac{1}{2k+1}
        \sum_{m=1}^{2k+1} \binom{2k+1}{m} \sin(x)^{2k+1-m} \left(-\frac{xy^2}{2} (1+G)\right)^{m}\\
    &\quad= x - \frac{xy^2}{2} (1+G) \left(1+\sum_{k=0}^\infty\frac{(2k+1)!!}{(2k+2)!!} \frac{1}{2k+3}
        \sum_{m=0}^{2k+3} \binom{2k+3}{m+1} \sin(x)^{2k+2-m} \left(-\frac{xy^2}{2} (1+G)\right)^{m}\right)\\
    &\quad= x- \frac{xy^2}{2} (1+o(\|(x,y)\|)).
\end{align*}
So, for
\begin{equation*}
    f(x,y):=\frac{x-\arcsin(\sin(x)\cos(y))}{xy^2} = \frac{1}{2} + o(\|(x,y)\|)\qquad \text{as $x,y\to 0$},
\end{equation*}
the total derivative of $f$ at $(x,y)=(0,0)$ exists and is given by
\begin{equation*}
    \dint f(0,0) = \lim_{(x,y)\to (0,0)} \frac{f(x,y)-\frac{1}{2}}{\|(x,y)\|} = 0.
\end{equation*}
In particular this implies that there exists $\varepsilon>0$ such that 
\begin{equation*}
    \frac{1}{2} - \varepsilon \leq f(x,y) \leq \frac{1}{2} + \varepsilon \qquad \text{for all } x,y\in (0,\varepsilon).
\end{equation*}

\medskip
Next, we show that 
\begin{equation*}
    f(x,y)\geq \frac{2}{\pi^2} \quad \text{for all $x\in[0,\pi/2]$ and $y\in[0,\pi]$.}
\end{equation*}
Let us fix $y_0\in(0,\pi)$ and consider function $g_{y_0}(x):=f(x,y_0)$. It is clear, that for $y_0=\pi/2$ we have $g_{y_0}(x)=4/\pi^2$ for any $x\in(0,\pi/2)$. In the next step we will show, that for $y_0\in (0,\pi/2)$ function $g_{y_0}(x)$ is strictly increasing and for $y_0\in(\pi/2,\pi)$ function $g_{y_0}(x)$ is strictly decreasing on the interval $x\in (0,\pi/2)$. 

Let us start with the case $y_0\in(0,\pi/2)$. Consider the derivative
\begin{equation*}
    g_{y_0}'(x)=-\frac{1}{x^2y_0^2}\Big(\frac{x\cos(x)\cos(y_0)}{\sqrt{1-\sin(x)^2\cos(y_0)^2}}-\arcsin(\sin(x)\cos(y_0))\Big).
\end{equation*}
It is clear, that $g_{y_0}(x)$ is strictly increasing on the interval $x\in (0,\pi/2)$ if and only if $g_{y_0}'(x)>0$ on this interval, which is equivalent to
\begin{equation*}
    \frac{x\cos(x)\cos(y_0)}{\sqrt{1-\sin(x)^2\cos(y_0)^2}}<\arcsin(\sin(x)\cos(y_0)).
\end{equation*}
Consider a change of variables $s:=\sin(x)\in(0,1)$, $t_0:=\cos(y_0)\in(0,1)$. Then the inequality above is equivalent to
\begin{equation}\label{eqn:minimum1}
    t_0\arcsin(s)\sqrt{1-s^2}<\arcsin(st_0)\sqrt{1-s^2t_0^2}.
\end{equation}
Denote by $r(z):=\arcsin(z)\sqrt{1-z^2}$ and by $q(t_0):=r(t_0s)-t_0r(s)$. Then \eqref{eqn:minimum1} is equivalent to $q(t_0)>0$ for $t_0\in(0,1)$. 
It is easy to verify that $q(0)=q(1)=0$ and
\begin{equation*}
    q''(t_0)=s^2r''(t_0s)<0,\qquad s,t_0\in(0,1),
\end{equation*}
since $r(z)$ is strictly concave on the interval $(0,1)$. This finishes the proof of the first statement.

In the case $y_0\in(\pi/2,\pi)$ we apply the change of variables $\tilde{y}_0:=\pi-y_0\in(0,\pi)$ and note that
\begin{equation*}
    g_{y_0}'(x)=-g_{\tilde{y}_0}'(x)<0,
\end{equation*}
as follows from the previous calculations. 
Thus,
\begin{equation*}
    \inf\limits_{x\in(0,\pi/2),y\in(0,\pi)} f(x,y)
    = \min \Big(\inf\limits_{y\in(0,\pi/2)}f(0,y), \inf\limits_{y\in(\pi/2,\pi)}f\big(\frac{\pi}{2},y\big), \frac{4}{\pi^2}\Big).
\end{equation*}
We calculate 
\begin{align*}
    \inf\limits_{y\in(0,\pi/2)}f(0,y)
    &=\inf\limits_{y\in(0,\pi/2)}\lim_{x\to 0}f(x,y)
    =\inf\limits_{y\in(0,\pi/2)} \frac{1-\cos(y)}{y^2}
    \geq \frac{1}{2},
\intertext{and}
    \inf\limits_{y\in(\pi/2,\pi)}f\big(\frac{\pi}{2},y\big)
    &=\inf\limits_{y\in(\pi/2,\pi)} \frac{2}{y\pi}
    =\frac{2}{\pi^2},
\end{align*}
which implies that $f(x,y)\geq 2/ \pi^2$ in the rectangle $(x,y)\in [0,\pi/2]\times [0,\pi]$.

\medskip

\subsection{Proof of \texorpdfstring{\eqref{eqn:trivial_lower}}{(5.2)} in Step 3} \label{apx:proof_trivial_lower}

We show that
\begin{equation*}
    \pi -x +\arcsin(\sin(x)\cos(y)) \geq \frac{1}{3} \left((\pi/2-x)+(\pi-y)\right)
        \qquad \text{for all } x\in(0,\pi/2), y\in(0,\pi).
\end{equation*}
Using the qualities $\sin x=\cos(\pi/2-x)$, $\cos y=-\cos(\pi-y)$, $\arcsin(-x)=-\arcsin x$ and making a change of variables $x=\pi/2-x$ and $y=\pi-y$ we see, that the above inequality is equivalent to
\begin{equation*}
    x + \pi/2 - \arcsin(\cos(x)\cos(y)) \geq \frac{1}{3} (x+y),
        \qquad \text{for all } x\in(0,\pi/2), y\in(0,\pi).
\end{equation*}
We use the estimate $\pi/2-\arcsin z \geq \sqrt{1-z}$ for $z\in[-1,1]$, and $\cos z\leq 1-z^2/5$ for all $z\in[0,\pi]$, to find that
\begin{align*}
    x+\pi/2-\arcsin(\cos(x)\cos(y)) 
    &\geq x + \sqrt{1-(1-x^2/5)(1-y^2/5)}\\
    &= x + \frac{1}{5}\sqrt{5(x^2+y^2)-x^2y^2} 
    \geq \frac{x+y}{3},\quad \text{for all } x\in(0,\pi/2), y\in(0,\pi),
\end{align*}
where the last inequality is an exercise.

\subsubsection*{Acknowledgement}
The authors started this project within a working group that formed during the workshop \textit{New Perspectives and Computational Challenges in High Dimensions} at the Mathematisches Forschungsinstitut Oberwolfach (MFO) and continued to work during the virtual Trimester Program \textit{The Interplay between High Dimensional Geometry and Probability} at the Hausdorff Research Institute for Mathematics (HIM). All support is gratefully acknowledged.\\ CT has also been supported by the DFG priority program SPP 2265 \textit{Random Geometric Systems}. AG was supported by the DFG under Germany's Excellence Strategy  EXC 2044 -- 390685587, \textit{Mathematics M\"unster: Dynamics - Geometry - Structure}. EMW is supported by NSF grant  DMS-2103482 and by Simons Fellowship 678608.

\end{document}